\documentclass[10pt]{article}

\RequirePackage{amsmath}
\RequirePackage{amsfonts}
\RequirePackage{amssymb}
\RequirePackage{hyperref}
\RequirePackage{multirow}
\RequirePackage{algorithm}
\RequirePackage{algpseudocode}
\RequirePackage{graphicx}
\RequirePackage{xcolor}
\RequirePackage{amsthm}
\RequirePackage{amsmath}
\RequirePackage{amsfonts}
\RequirePackage{amssymb}
\RequirePackage{hyperref}
\RequirePackage{multirow}
\RequirePackage{algorithm}
\RequirePackage{algpseudocode}
\RequirePackage{graphicx}
\RequirePackage{xcolor}
\allowdisplaybreaks

\usepackage[numbers,sort&compress]{natbib}
\usepackage{dpr_fec,ifthen,dsfont}
\usepackage{notations}
\theoremstyle{plain}
\newtheorem{theorem}{Theorem}[section]
\newtheorem{lemma}[theorem]{Lemma}

\theoremstyle{definition}
\newtheorem{definition}[theorem]{Definition}

\newtheorem{assumption}[theorem]{Assumption}

\theoremstyle{remark}

\newtheorem{strat}{Strategy}

\usepackage{isetechreport}

\def\revisionno{2}
\def\originaldate{November 4, 2022}
\coraltrue
\cvcrfalse

\begin{document}

\title{Inexact Proximal-Gradient Methods with Support Identification


}

\author{Yutong Dai\thanks{E-mail: yud319@lehigh.edu}}
\affil{Department of Industrial and Systems Engineering, Lehigh University}
\author{Daniel P.~Robinson\thanks{E-mail: daniel.p.robinson@lehigh.edu}}
\affil{Department of Industrial and Systems Engineering, Lehigh University}
\titlepage

\maketitle

\begin{abstract}
  We consider the proximal-gradient method for minimizing an objective function that is the sum of a smooth function and a non-smooth convex function. A feature that distinguishes our work from most in the literature is that we assume that the associated proximal operator does not admit a closed-form solution. To address this challenge, we study two \emph{adaptive} and \emph{implementable} termination conditions that dictate how accurately the proximal-gradient subproblem is solved.  We prove that the number of iterations required for the inexact proximal-gradient method to reach a $\tau > 0$ approximate first-order stationary point is $\Ocal(\tau^{-2})$, which matches the similar result that holds when exact subproblem solutions are computed.  Also, by focusing on the overlapping group $\ell_1$ regularizer, we propose an algorithm for approximately solving the proximal-gradient subproblem, and then prove that its iterates identify (asymptotically) the support of an optimal solution.  If one imposes additional control over  the accuracy to which each subproblem is solved, we give an upper bound on the maximum number of iterations before the support of an optimal solution is obtained.

\end{abstract}

\section{Introduction}\label{sec.introduction}

We consider the minimization of a function written as the sum of a smooth function $f:\R{n} \to \R{}$ and a non-smooth convex function $r:\R{n} \to \R{}$, which can be written as
\begin{equation}\label{prob.main}
\min_{x\in\R{n}} \ f(x) + r(x).
\end{equation}
Many algorithms exist for solving such problems (e.g., see~\cite{bach2012optimization,beck2009fast}), which cover a wide range of machine learning problems.  A common choice for the regularizer $r$ is the $\ell_1$ norm~\citep{tibshirani1996regression}, which is used to attain solutions to problem~\eqref{prob.main} that have few nonzero entries, i.e., a sparse solution. Computing such a solution is critical in machine learning applications that use model prediction since it is only the nonzero entries in the solution that define the model, and therefore sparse solutions define simpler and easier to interpret and understand models. Despite the successes of $\ell_1$-norm regularization, its inadequacy for some modern applications has been observed. For example, since covariates often come in groups (e.g., genes that regulate hormone levels), one may wish to select them jointly. Also, incorporating group information into the modeling process can improve both the interpretability and accuracy of the resulting model~\citep{zeng2016overlapping}.  Therefore, new regularizers like the elastic-net~\citep{zou2005regularization} and group lasso~\citep{yuan2006model} have been proposed to address these issues, and have been successful in practice for genome-wide SNP selection, predicting Parkinson's disease, and gene microarray selection~\citep{ayers2010snp,zhu2004classification,cawley2006gene,shen2019use}.

The proximal-gradient (PG) method and its accelerated variants (APG) (e.g., see~\cite{beck2009fast}) form an important class of algorithms for solving problems of the form given by~\eqref{prob.main}. The core computation in a PG method is the evaluation of the proximal operator associated with the regularizer at a given point.  In particular, given a proximal parameter $\alpha>0$, the proximal operator associated with $r(\cdot)$ evaluated at $u\in\R{n}$ is defined as the unique solution to the following optimization problem:
\begin{equation}\label{prob.prox}
  \prox{\alpha r}{u} := \argmin{x\in\R{n}}
  \ \tfrac{1}{2\alpha}\norm{x-u}_2^2 +  r(x).
\end{equation}
The most basic PG method sets $u \gets x-\alpha\grad f(x)$, where $x$ is the current estimate of a solution to problem~\eqref{prob.main}, so that the PG method update takes the form $x \gets \prox{\alpha r}{x-\alpha\grad f(x)}$. For relatively simple regularizers, such as the $\ell_1$ norm, the solution to~\eqref{prob.prox} can be computed in closed-form.  However, for more complicated regularizers, such as the overlapping group $\ell_1$ norm \citep{yuan2013efficient}, the solution to~\eqref{prob.prox} does not admit a closed-form solution, and therefore an iterative algorithm (the iterations of which are often called the \emph{inner iterations}) is required to \emph{approximately} solve problem~\eqref{prob.prox} each iteration.  As mentioned in~\cite{schmidt2011}, many numerical experiments have shown that PG methods often work well even when the proximal operator is computed inexactly by approximately solving~\eqref{prob.prox}. Consequently, research has been devoted to understanding how inexact evaluations of the proximal operator impact the convergence properties of the PG method; we discuss these efforts next.

\subsection{Related Work}

Since the PG method solves problem~\eqref{prob.main} by solving a sequence of subproblems of the form~\eqref{prob.prox}, to guarantee convergence of an \emph{inexact} PG method, the errors when approximately solving~\eqref{prob.prox} must be carefully managed.  There are two basic types of strategies for controlling errors, an \emph{absolute criterion} and an \emph{adaptive criterion}.

Following the convention in~\cite{lin2018catalyst}, an \textit{absolute criterion} defines \emph{in advance} the error tolerance for subproblem~\eqref{prob.prox} that is acceptable during each iteration. For example, a framework is analyzed in~\cite{schmidt2011} that shows that if the error in solving~\eqref{prob.prox} is decreased at the rate of $\Ocal(1/k^\delta)$, where $k$ is the iteration counter and $\delta>2$ (respectively, $\delta>4$) for the PG (respectively, APG) method, then the inexact PG (respectively, APG) method shares the same convergence rate as its exact counterpart. In that work, for $\epsilon > 0$, the authors define $\xhat$ as an $\epsilon$-approximate solution to subproblem~\eqref{prob.prox} if and only if it satisfies $\tfrac{1}{2\alpha}\|\xhat-u\|_2^2 + r(\xhat) 
\leq \min_{x\in\R{n}} \left(\tfrac{1}{2\alpha}\|x-u\|^2_2 + r(x)\right) + \epsilon$. Other papers have considered different ways of characterizing an inexact solution to subproblem~\eqref{prob.prox}.  For example, \cite[Definition 2.1]{villa2013accelerated} defines $\xhat$ to be an inexact solution if and only if $(u-\xhat)/\alpha\in \partial_{\epsilon}r(\xhat)$ for an appropriate choice of $u$ (see Definition~\ref{def:epssubdiff} for the definition of the $\epsilon$-subdifferential $\partial_{\epsilon}$), and then use this notion to analyze the convergence rate of the APG method.  In~\cite[Definition 2.5]{rasch2020inexact}, $\xhat$ is considered an inexact solution to~\eqref{prob.prox} if and only if $\xhat=\prox{\alpha r}{u+e}$ for an appropriate choice of $u$ and where $\norm{e}\leq \sqrt{2\alpha\epsilon}$, and then  analyze the convergence rate of primal-dual algorithms for solving a saddle-point problem. These latter two notions are more stringent than the notion in~\cite{schmidt2011}, and all three are different in how the error is decomposed among the $\epsilon$-subdifferential of $r(x)$ and $\tfrac{1}{2\alpha}\norm{x-u}^2$ (for more details, see the discussion in~\cite[Lemma 2.6]{rasch2020inexact}). 

The motivation for an \textit{adaptive criterion} is to produce, given the current iterate $x_k$, an approximate solution $x_{k+1}$ to subproblem~\eqref{prob.prox} such that $f(x_{k+1}) + r(x_{k+1})$ is ``significantly" less than $f(x_k)+r(x_k)$.  A criterion of this type, by its nature, cannot be define in advance (as is the case with an absolute criterion) because it must take information about $f+r$ at $x_k$ into account, thus making it adaptive. Moreover, as argued in \cite{fuentes_descentwise_2012}, an absolute criterion may not be ideal for overall efficiency since the actual goal is not to find $x_{k+1} \approx \prox{\alpha_kr}{u}$, but rather to find an $x_{k+1}$ that significantly reduces the objective function.  The authors in~\cite{fuentes_descentwise_2012} consider problem~\eqref{prob.main} for  the special case when $f$ is ill-conditioned and $r \equiv 0$, in which case the method is equivalent to an inexact proximal-point method.  In particular, given $x_k$ and $\alpha_k>0$, they approximately solve $\min_{x\in\R{n}} \big( \frac{1}{2\alpha_k}\norm{x-x_k}_2^2 + f(x)\big)$, where $x_{k+1}$ is acceptable as an inexact solution if $f(x_k)-f(x_{k+1})\geq \eta\norm{\grad f(x_{k+1})}_2\norm{x_k-x_{k+1}}_2$ with $\eta\in (0,1)$. This work is extended in~\cite{lin2018catalyst} to allow for nontrivial $r$ and acceleration by approximately solving $\min_{x\in\R{n}} \big( \frac{1}{2\alpha_k}\norm{x-x_k}_2^2 + f(x) + r(x)\big)$, where $x_{k+1}$ is an acceptable inexact solution if $\tfrac{1}{2\alpha_k}\norm{x_{k+1}-x_k}_2^2 + f(x_{k+1}) + r(x_{k+1}) - \min_{x\in\R{n}} \big( \tfrac{1}{2\alpha}\norm{x-x_k}_2^2 +  f(x)+r(x)\big) \leq \kappa \norm{x_{k+1} - x_k}_2^2$ for some $\kappa>0$. Relating a measure of error in the subproblem to $\Ocal(\norm{x_{k+1} - x_k}_2)$ dates back to \cite{rockafellar1976monotone} and has been extensively studied for the proximal-point method.  To integrate second-derivative information, \cite{lee2019inexact} proposed to approximately solve, for a symmetric matrix $H$, the subproblem 
\begin{equation}\label{prob-prox-newton}
    \min_{x\in \R{n}} \ \nabla f(x_k)^T (x-x_k) + \tfrac{1}{2}(x-x_k)^TH(x-x_k) + r(x)-r(x_k).
\end{equation}
When $H=(1/\alpha_k)I$ and $u = x_k - \alpha_k\nabla f(x_k)$, subproblem~\eqref{prob-prox-newton} recovers subproblem~\eqref{prob.prox}. The vector $x_{k+1}$ is considered an acceptable approximate solution to subproblem~\eqref{prob-prox-newton} if it achieves a decrease in its objective function that is at least some fraction of that achieved by its exact solution. Verifying this adaptive criterion requires either estimating a tight lower bound of the minimal value for subproblem \eqref{prob-prox-newton} or devising (when problem~\eqref{prob-prox-newton} is strongly convex) an algorithm with global linear convergence for~\eqref{prob-prox-newton}. Recently, \cite{rasch2020inexact} designed two adaptive criteria specifically tailored to FISTA \citep{beck2017first}.

When $r$ is a sparsity-inducing regularizer, one is interested in a support identification property (i.e., the iterates of an algorithm correctly identify the support of an optimal solution to problem~\eqref{prob.main} in finitely many iterations).  It is shown in~\cite{nutini2019active,sun19a} that when problem \eqref{prob.prox} is solved \emph{exactly}, the PG and APG methods have the support identification property under non-degeneracy assumptions.\footnote{Support identification is called active-set identification in \cite{nutini2019active} and  manifold identification in \cite{sun19a}.} While support identification results are proved in~\cite{lee2020accelerating} for a framework built upon subproblem~\eqref{prob-prox-newton}, they either require the explicit computation of the exact solution to~\eqref{prob.prox} or require its implicit computation by making an assumption on the inexact subproblem solution obtained. To the best of our knowledge, no research has explored the support identification property when problem \eqref{prob.prox} must be solved \emph{inexactly}, which is a unique aspect of our paper. 

\subsection{Contributions}

This paper makes the following contributions in the area of PG methodology.

\begin{itemize}
    \item We propose an \emph{inexact} PG framework that allows for two practical \emph{adaptive criteria} for determining acceptable approximate solutions to the PG subproblem. The adaptive criteria are easier to verify compared to most in the literature.  For example, compared with~\cite{lin2018catalyst}, our criteria are easier to verify for regularizers such as the overlapping \grplone{}; to verify the adaptive criterion in \cite{lin2018catalyst}, one must either evaluate the proximal operator of $r$ exactly, which excludes the overlapping \grplone{}, or estimate a tight lower-bound of $\min_{x\in\R{n}} \frac{1}{2\alpha_k}\norm{x-u}^2 + f(x)+ r(x)$, which can be as difficult as solving problem~\eqref{prob.main} itself. 
    \item We provide a unified worst-case complexity bound analysis of our PG framework noted in the previous bullet point. Interestingly, the complexity result for our \emph{inexact} PG framework matches the result that holds for the \emph{exact} PG method.  
    \item For the special case of the overlapping \grplone{} regularizer, whose proximal operator does not admit a closed-form solution, we propose an enhanced PG subproblem solver that is designed with support identification in mind. When our PG framework uses our PG subproblem solver, we prove (under common assumptions) a worst-case complexity result for the number of iterations until optimal support identification occurs. To the best of our knowledge, this is the first work that establishes such a result for a framework built upon implementable (verifiable) conditions when the regularizer does not have a closed-form solution.  

\end{itemize}

\subsection{Notation and assumptions}

Let $\R{}$ denote the set of real numbers, $\R{n}$ denote the set of $n$-dimensional real vectors, and $\R{m \times n}$ denote the set of $m$-by-$n$-dimensional real matrices. Let $\norm{\cdot}$ denote the $\ell_2$ norm. The set of nonnegative integers is  denoted as $\N{} := \{0,1,2,\dots\}$ and for any positive integer $n$, we define $[n] := \{1,2,\dots,n\}$. For a matrix $A\in\R{m\times n}$ and index sets $(I,I')\subseteq [m]\times [n]$, we let $A_{[I,:]}$ denote the sub-matrix of $A$ formed by taking rows from the index set $I$ and all columns, and $A_{[I,I']}$ denote the sub-matrix of $A$ formed by taking rows in the index set $I$ and columns in the index set $I'$. For a convex function $h: \R{n}\to \R{}$, denote its Fenchel conjugate as $h^*$. Given $x\in\R{n}$ and $d\in\R{n}$, the directional derivative of $h$ at $x$ in the direction $d$ is denoted by $D_h(x;d)$.  Given a set $\Ccal\subseteq\R{n}$ that is closed and convex, define the projection operator $\proj{\Ccal}{\cdot}: \R{n} \to \Ccal$ as
$$
\proj{\Ccal}{\ybar} 
:= \argmin{y\in\Ccal} \ \|y-\ybar\|_2,
$$
which then also allows us to define the distance of a vector $\ybar\in\R{n}$ to the set $\Ccal$ as
$$
\dist(\ybar,\Ccal)
:= \| \ybar - \proj{\Ccal}{\ybar} \|.
$$

For any $\xbar \in\R{n}$ and $\alphabar>0$, we define the PG update as
\begin{equation}
\T(\xbar,\alphabar) 
:= \argmin{x\in\R{n}} \ \phi(x;\xbar,\alphabar)
\ \ \ \text{with} \ \ \
\phi(x;\xbar, \alphabar) := \tfrac{1}{2\alphabar}\norm{x- \big(\xbar-\alphabar \grad  f(\xbar)\big)}^2 + r(x).\label{def:pg-update}
\end{equation}
Using the notion of $\epsilon$-inexactness from~\cite[equation (4)]{schmidt2011}, we use $\Teps(\xbar,\alphabar)$ to denote the \emph{set} of $\epsilon$-PG updates, which is defined by 
\begin{equation}
 \Teps(\xbar,\alphabar)
 :=\{\xhat\in\R{n}~|~\phi(\xhat;\xbar, \alphabar) \leq \phi(\T(\xbar, \alphabar);\xbar, \alphabar) + \epsilon\}.  
 \label{def:ipg-update}
\end{equation}
Note that it follows from these definitions that $\Tcal_0(\xbar,\alphabar) = \{\T(\xbar,\alphabar)\}$.

Finally, we make the following assumption throughout the paper.

\bassumption\label{ass.first}
Define $\Lcal := \{x \in\R{n} : f(x) + r(x) \leq f(x_{0}) + r(x_0)\}$ where $x_0$ is an initial solution estimate to~\eqref{prob.main}. The function $f$ is continuously differentiable and $\grad f$ is Lipschitz continuous with constant $L_g$ on an open neighborhood containing $\Lcal$. The function $r$ is closed, proper, and convex. The function $f+r$ is bounded below.
\eassumption

\subsection{Organization}

In Section~\ref{sec.algorithm}, we describe the inexact PG method that uses adaptive termination criteria for solving problem~\eqref{prob.main}.  The convergence analysis and support identification result is presented in Section~\ref{sec.analysis} and Section~\ref{sec.identification},  respectively.
Numerical results are given in Section~\ref{sec.numerical}, and concluding remarks are presented in Section~\ref{sec.conclusion}.

\section{Algorithm}\label{sec.algorithm}

Our proposed inexact PG method is stated as Algorithm~\ref{alg:main}.  Given the $k$th iterate $x_k$ and proximal parameter $\alpha_k$, the algorithm computes an $\epsilon_k$-PG update $\xhat_{k+1}\in\Tepsk(x_k,\alpha_k)$ and associated step $s_k = \xhat_{k+1} - x_k$, where the value for $\epsk$ depends on the value of the input parameter $\opt\in\{\optone,\opttwo\}$. In particular, if $\opt\ = \optone{}$, then $\epsk = c_k\|s_k\|^2$ with $c_k$ chosen in Step~\ref{line:ck}, while if $\opt\ = \opttwo{}$, then  $\epsk = \gamma_2 \big(\phi(x_k;x_k, \alpha_k) - \phi(\T(x_k,\alpha_k);x_k,\alpha_k)\big)$ with $\gamma_2 \in (0,1/2]$. Next, $\Delta_k$ is set to a value that depends on the value of $\opt{}$ (see Steps~\ref{line:Deltak1}   and~\ref{line:Deltak2}). The choices for $\Delta_k$ and $c_{k}$ and the restrictions on $\gamma_1$ and $\gamma_2$ in Algorithm~\ref{alg:main} ensure that $s_k$ is a descent direction for $f+r$ at $x_k$ (see Lemma~\ref{lemma:per-iter-gain}), and that the line search performed in Steps~\ref{line:beginwhile}--\ref{line:endwhile} gives a decrease in the objective function proportional to $\alpha_k\chipgk{^2}$ when the unit step is accepted (see Lemma~\ref{lem:K-pg}(i)) where
\begin{equation}\label{def:chipgk}
\chipgk 
:= \norm{T(x_k,\alpha_k)-x_k}/\alpha_k
\end{equation}
is a measure of first-order optimality for problem~\eqref{prob.main} at $x_k$ (see~\cite[Theorem~10.7]{beck2017first} for more details on $\chipgk$).  Finally, in Steps~\ref{line:update-pg-if}--\ref{line:update-pg-endif} the proximal parameter for the next iteration is set to a fraction of its current value  anytime $j > 0$ (i.e., anytime the unit step size is not accepted during the linesearch).

To successfully implement Algorithm~\ref{alg:main}, for each $k\in\N{}$, one must be able to compute the triple $(\xhat_{k+1},s_k,\epsk)$ needed in Steps~\ref{line:triple1} and~\ref{line:triple2}.  As long as $x_k$ is not optimal (i.e., as long as $\chipgk > 0$), the existence of such points is guaranteed, although computing one is challenging for nontrivial regularizers. We consider this practical question in Section~\ref{sec.identification}. 

\balgorithm[!th]
\caption{Inexact PG Method for solving problem~\eqref{prob.main}.}
\label{alg:main}
\balgorithmic[1]
   \State \textbf{Input:} Initial estimate $x_0\in\R{n}$ and $\opt\in\{\optone,\opttwo\}$.
   \State \textbf{Constants:} $\{\xi,\eta,\zeta\}\subset (0,1)$, $\alpha_0 \in (0,\infty)$, $\gamma_1\in(0,2)$, and $\gamma_2\in(0,1/2]$
   \For{$k = 0,1,2,\dots$}
      \If{ $(\opt{} = \optone{})$}
          \State Choose $c_k\in \Big(0, \frac{1}{4}\left(\sqrt{\frac{6}{(1+\gamma_1)\alpha_k}}-\sqrt{\frac{2}{\alpha_k}}\right)^2\Big]$.\label{line:ck}
          \State\label{line:triple1} Compute a triple $(\xhat_{k+1},s_k,\epsk)$ satisfying the conditions
          $$
          \xhat_{k+1} \in \Tepsk(x_k,\alpha_k), \ \
          s_k = \xhat_{k+1} - x_k, \ \ \text{and} \ \
          \epsk = c_k\|s_k\|^2.
          $$
          \State Set $\Delta_k \gets -  \tfrac{1}{\alpha_k}\norm{s_k}^2 + \sqrt{(2/\alpha_k)\epsk}\norm{s_k} + \epsk$.\label{line:Deltak1}
       \Else{ $(\opt{} = \opttwo{})$}
          \State\label{line:triple2} Compute a triple $(\xhat_{k+1},s_k,\epsk)$ satisfying the conditions 
          $$
          \quad \xhat_{k+1} \in \Tepsk(x_k,\alpha_k), \ \ 
          s_k = \xhat_{k+1} - x_k, 
          \ \ \text{and} \ \
          \epsk = \gamma_2 \big(\phi(x_k;x_k, \alpha_k) - \phi(\T(x_k,\alpha_k);x_k,\alpha_k)\big).
          $$
          \State Set $\Delta_k \gets r(x_k+s_k) - r(x_k) + \grad  f(x_k)^T s_k$.\label{line:Deltak2}
        \EndIf
        \State Set $j \gets 0$ and $\xtrial_{k+1,0} \gets x_k + s_k$.\label{line:y0-pg}
            \While{$f(\xtrial_{k+1,j}) + r(\xtrial_{k+1,j}) > f(x_k) + r(x_k) + \eta\xi^j\Delta_k$}\label{line:armijo-pg}\label{line:beginwhile}
              \State Set $j \gets j + 1$ and $\xtrial_{k+1,j} \gets x_k + \xi^j s_k$. \label{line:y-up-pg}
            \EndWhile\label{line:endwhile}
            \State $x_{k+1} \gets \xtrial_{k+1,j}$\label{line:next-iterate}
            \If{$j = 0$} \label{line:update-pg-if}
                \State $\alpha_{k+1}\gets \alpha_k$ and $\flagpgk \gets \samealpha$. \label{line:update2}
            \Else
                \State $\alpha_{k+1}\gets \zeta\alpha_k$ and $\flagpgk \gets \decalpha$. \label{line:update2b}
            \EndIf \label{line:update-pg-endif}
            \EndFor
        \ealgorithmic
\ealgorithm

\section{Convergence Analysis}\label{sec.analysis}

In this section we analyze Algorithm~\ref{alg:main}. 
In particular, in Section~\ref{sec.preliminaries}, we give preliminary results related to $\epsilon$-PG updates, and in Section~\ref{sec.complexity} we present a complexity analysis.

\subsection{Preliminary results}\label{sec.preliminaries}

In this section, we discuss the $\epsilon$-subdifferential of a function, and characterize the properties of the $\epsilon$-PG update.  We begin with the definition of the $\epsilon$-subdifferential.

\begin{definition}[$\epsilon$-subdifferential]{\cite[Section 4.3]{bertsekas2003convex}}\label{def:epssubdiff}
For a given convex function $h: \R{n}\to \R{}$, 
vector $x\in\R{n}$, and scalar $\epsilon>0$, the $\epsilon$-subdifferential of $h$ at $x\in\R{n}$ is 
$$
\partial_{\epsilon}h(x)
:=\{d\in\R{n}~|~ h(z)\geq h(x) + d^T(z-x) - \epsilon \text{ for all }z\in\R{n}\}.
$$
\end{definition}

We now bound the difference between an $\epsilon$-PG update and the exact PG update. 

\begin{lemma}[\cite{salzo2012inexact}]\label{lemma:chibound}
Let $\xbar \in\R{n}$, $\alphabar>0$, and $\epsilon>0$. It holds that
$$
\norm{\T(\xbar,\alphabar) - \xhat}^2 \leq 2\alphabar\epsilon
\ \ \text{for all} \ \ \xhat\in\Teps(\xbar,\alphabar).
$$
\end{lemma}
\begin{proof}
Note that $\phi(x ;\xbar,\alphabar)$ (see~\eqref{def:pg-update}) is strongly convex as a function of $x$ with strong convexity parameter greater than or equal to $1/\alphabar$. Therefore, we may conclude that 
$$
\begin{aligned}
     \epsilon \geq \phi(\xhat;\xbar, \alphabar)- \phi(\T(\xbar,\alphabar);\xbar, \alphabar) \geq \tfrac{1}{2\alphabar}\norm{\T(\xbar,\alphabar) - \xhat}^2
     \ \ \text{for all $\xhat\in\Teps(\xbar,\alphabar)$}
\end{aligned}
$$
where the last inequality follows from $0\in\partial \phi(T(\xbar,\alphabar);\xbar,\alphabar)$ (i.e., it follows from the optimality conditions for problem~\eqref{def:pg-update}).
\end{proof}

Part (i) of the next result constructs two upper-bounds on the directional derivative of the objective function $f+r$ at $\xbar$ along the direction associated with the $\epsilon$-PG update. This result suggests that by choosing $\epsilon$ sufficiently small, the directional derivative will be negative.  For additional insight, the reader should compare them to the definition of $\Delta_k$ in lines~\ref{line:Deltak1} and~\ref{line:Deltak2} of Algorithm~\ref{alg:main}.  The second part of the next result bounds the change in $f+r$ that results from taking the step along the direction. 

\begin{lemma}\label{lem:epsilon-descent-full}
Let $\xbar \in\R{n}$, $\alphabar > 0$, and $\epsilon > 0$. If $\xhat \in \Teps(\xbar,\alphabar)$, then $s = \xhat-\xbar$ satisfies the following results.
\begin{itemize}
\item[(i)] The directional derivative of $f+r$ at $\xbar$ in the direction $s$
satisfies
$$
D_{f+r}(\xbar;s) 
\leq -\tfrac{1}{\alphabar}\norm{s}^2 
+ \sqrt{\tfrac{2\epsilon}{\alphabar}}\norm{s}+\epsilon 
\ \ \
\text{and}
\ \ \ 
D_{f+r}(\xbar;s) \leq r(\xbar + s) - r(\xbar) + \grad f(\xbar)^T s.
$$
\item[(ii)] A bound on the change in the objective function after taking the step $s$ is
$$
f(\xbar + s) + r(\xbar + s)  \\
\leq f(\xbar) + r(\xbar) 
- (\tfrac{1}{\alphabar}-\tfrac{L_g}{2})\norm{s}^2 + \sqrt{\tfrac{2\epsilon}{\alphabar}}\norm{s}+ \epsilon.
$$
\end{itemize}
\end{lemma}
\begin{proof}
Let $\xhat\in \Teps(\xbar,\alphabar)$. It follows from~\cite[Lemma 2]{schmidt2011} that there is a $w\in\R{n}$ so that
\begin{equation}\label{epsilon-subgradient}
 \norm{w}\leq \sqrt{2\alphabar\epsilon}
 \ \ \text{and} \ \ 
 g_\epsilon := \frac{1}{\alphabar}\left[\xbar-\alphabar\grad  f(\xbar) + w - \xhat\right] \in\partial_\epsilon r(\xhat).
\end{equation}
From $g_\epsilon\in\partial_\epsilon r(\xhat)$ and convexity of $r(x)$, 
it follows that
\begin{equation}\label{descent-2}
r(\xbar)\geq r(\xhat) + g_\epsilon^T(\xbar - \xhat)-\epsilon
\ \ \text{and} \ \
r(\xhat)\geq r(\xbar) + g_r^T(\xhat - \xbar) 
\ \ \text{for all $g_r\in\partial r(\xbar)$.}
\end{equation}

Let us now proceed to prove part (i). Adding the two inequalities in~\eqref{descent-2} together gives
$(g_r-g_\epsilon)^T(\xhat - \xbar)\leq \epsilon$ for all $g_r\in\partial r(\xbar)$.  This inequality may then be combined with $s = \xhat-\xbar$, the definition of $g_\epsilon$,  the Cauchy-Schwarz inequality, and \eqref{epsilon-subgradient} to obtain
\bequation\label{eq:s-inner-g}
\begin{aligned}
s^T(\grad  f(\xbar) + g_r)
&= (\xhat - \xbar)^T(\grad  f(\xbar) + g_r) 
= (\xhat - \xbar)^T\big(\tfrac{1}{\alphabar}(\xbar-\xhat+w) + g_r - g_\epsilon\big) \\
&= -\tfrac{1}{\alphabar}\norm{\xhat - \xbar}^2 + \tfrac{1}{\alphabar}(\xhat - \xbar)^Tw + (\xhat - \xbar)^T(g_r - g_\epsilon)\\
&\leq -\tfrac{1}{\alphabar}\norm{s}^2 + \tfrac{1}{\alphabar}\sqrt{2\alphabar\epsilon}\norm{s}+ \epsilon
\ \ \text{for all $g_r\in \partial r(\xbar)$.}
\end{aligned}
\eequation
By the definition of the directional derivative, the fact that $f(x)$ is differentiable, convexity of $r$, the result in~\cite[Theorem~2.87]{mordukhovich2013easy}, and \eqref{eq:s-inner-g}, it follows that
$$
D_{f+r}(\xbar;s) 
=  s^T\grad  f(\xbar) + \sup_{g_r\in \partial r(\xbar)}s^Tg_r \\
\leq -\tfrac{1}{\alphabar}\norm{s}^2 + \sqrt{\tfrac{2\epsilon}{\alphabar}}\norm{s}+ \epsilon,
$$
which completes the proof of the first inequality in part (i). The proof of the second inequality in part (i) follows directly from the proof in~\cite[Lemma 1]{lee2019inexact}.

To prove part (ii), we can use the Lipschitz continuity of $\nabla f$ with Lipschitz constant $L_g$ (see Assumption \ref{ass.first}), $\xhat = \xbar + s$, and the first  inequality in~\eqref{descent-2}
to obtain
\begin{align*}
    f(\xbar + s) + r(\xbar + s)
     &\leq f(\xbar) + \grad  f(\xbar)^Ts  + \tfrac{L_g}{2}\norm{s}^2 + r(\xbar) + g_\epsilon^Ts+\epsilon\\
     &= f(\xbar) + r(\xbar) + (\grad  f(\xbar)+g_\epsilon)^Ts  + \tfrac{L_g}{2}\norm{s}^2 +\epsilon.
\end{align*}
Combining this inequality with \eqref{epsilon-subgradient} and the Cauchy-Schwarz inequality gives
\begin{align*}
    f(\xbar + s) + r(\xbar + s)
     &\leq f(\xbar) + r(\xbar) -\tfrac{1}{\alphabar}(s - w)^Ts  + \tfrac{L_g}{2}\norm{s}^2 +\epsilon\\
     &\leq f(\xbar) + r(\xbar) - \left(\tfrac{1}{\alphabar}-\tfrac{L_g}{2}\right)\norm{s}^2  + \tfrac{1}{\alphabar}\norm{w}\norm{s} +\epsilon\\
     &\leq f(\xbar) + r(\xbar) - \left(\tfrac{1}{\alphabar}-\tfrac{L_g}{2}\right)\norm{s}^2  + \tfrac{1}{\alphabar}\sqrt{2\alphabar\epsilon}\norm{s} +\epsilon,
\end{align*}
which completes the proof.
\end{proof}

\subsection{Global Complexity}
\label{sec.complexity}

In this section, we analyze the worst-case iteration complexity of Algorithm~\ref{alg:main} for identifying an approximate  stationary point of problem~\eqref{prob.main}.  Specifically, given a tolerance $\tau \in (0,1]$, we derive an upper bound on the number of iterations until  $\chipgk \leq \tau$. We will assume throughout, for each $k\in\N{}$, that $\chipgk > 0$ since if $\chipgk = 0$ then $x_k$ would be an \emph{exact} first-order solution. The next result gives an upper bound on $\Delta_k$, which in turn gives an upper bound on the directional derivative of $f+r$ at $x_k$ in the direction $s_k$.

\begin{lemma}\label{lemma:per-iter-gain}
For each $k\in\N{}$, the directional derivative of $f+r$ satisfies 
\begin{equation}\label{def:betak}
D_{f+r}(x_k;s_k) 
\leq \Delta_k
\leq
-\beta\alpha_k \chipgk^2
< 0
\ \ \ \text{where} \ \ \
\beta := 
\begin{cases}
\frac{\gamma_1}{2} & \text{if $\opt{} = \optone{}$,}\\
\frac{1}{4} & \text{if $\opt{} =  \opttwo{}$}.
\end{cases}
\end{equation}
\end{lemma}
\begin{proof}
We start by discussing properties of $\xhat_{k+1}$ and $\epsk$ that hold regardless of the value of $\opt$. It follows from~\eqref{def:chipgk} and the triangle inequality that $\xhat_{k+1}$ satisfies
\begin{align}
     \chipgk^2 
     &=\frac{\norm{\T(x_k,\alpha_k)-\hat x_{k+1}+\hat x_{k+1}-x_k}^2}{ \alpha_k^2}\label{eq:itermediate} \\
     & \leq \frac{\norm{\T(x_k,\alpha_k) - \hat x_{k+1}}^2}{\alpha_k^2} + \frac{2\norm{\T(x_k,\alpha_k) - \hat x_{k+1}}\norm{\hat x_{k+1}-x_{k}}}{\alpha_k^2} + \frac{\norm{\hat x_{k+1}-x_{k}}^2}{\alpha_k^2}. \nonumber
\end{align}
Also, since $\xhat_{k+1}\in\Tepsk(x_k,\alpha_k)$, we can apply Lemma~\ref{lemma:chibound} with $\xhat=\xhat_{k+1}$, $\xbar=x_{k}, \alphabar=\alpha_k$, and $\epsilon=\epsk$ to get
\begin{equation}\label{ineq:Terr}
\|\T(x_k,\alpha_k)-\xhat_{k+1}\|^2 
\leq 2\alpha_k\epsk.
\end{equation}
Now, let us consider the two cases determined by the value of the parameter $\opt{}$.

\textbf{Case 1 ($\opt{} = \optone{}$).}
It follows from~\eqref{eq:itermediate},  \eqref{ineq:Terr}, and $s_k = \xhat_{k+1} - x_k$ that
\begin{equation}\label{chik-ub}
\alpha_k\chipgk{^2} 
\leq 2\epsk + \frac{2\sqrt{2\alpha_k\epsk}\norm{s_k}}{\alpha_k} + \frac{\norm{s_k}^2}{\alpha_k}.
\end{equation}
Next, it follows from Step~\ref{line:triple1} and how $c_k$ is chosen in Step~\ref{line:ck} that
$$
\epsk 
= c_k\|s_k\|^2
\in \Big(0, \tfrac{1}{4}\left(\sqrt{\tfrac{6}{(1+\gamma_1)\alpha_k}}-\sqrt{\tfrac{2}{\alpha_k}}\right)^2\|s_k\|^2\Big]. 
$$
Using the fact that $\epsk$ lies in this range, it can be verified that
\begin{equation}\label{eq:lb-connection}
\tfrac{\gamma_1}{2\alpha_k} \left( 2\alpha_k\epsk + 2\sqrt{2\alpha_k\epsk}\norm{s_k}+\norm{s_k}^2\right) 
\leq  \tfrac{1}{\alpha_k}\norm{s_k}^2 - \sqrt{\tfrac{2\epsk}{\alpha_k}}\norm{s_k}-\epsk.
\end{equation}
Next, it follows from Step~\ref{line:Deltak1}, \eqref{eq:lb-connection}, and~\eqref{chik-ub} that
\begin{align}
\Delta_k 
&= -\left( \tfrac{1}{\alpha_k}\norm{s_k}^2 - \sqrt{\tfrac{2\epsk}{\alpha_k}}\norm{s_k}-\epsk \right) \nonumber 
\leq -\tfrac{\gamma_1}{2\alpha_k} \left( 2\alpha_k\epsk + 2\sqrt{2\alpha_k\epsk}\norm{s_k}+\norm{s_k}^2\right) \\
&= -\tfrac{\gamma_1}{2} \left( 2\epsk + \tfrac{2\sqrt{2\alpha_k\epsk}\norm{s_k}}{\alpha_k} + \tfrac{\norm{s_k}^2}{\alpha_k}\right) 
\leq -\tfrac{\gamma_1}{2}\alpha_k\chipgk{^2}. \label{Deltak-bd}
\end{align}
Finally, we can combine the equality in~\eqref{Deltak-bd} with Lemma~\ref{lem:epsilon-descent-full}(i) to obtain
$$
D_{f+r}(x_k;s_k) \leq -\left( \tfrac{1}{\alpha_k}\norm{s_k}^2 - \sqrt{\tfrac{2\epsk}{\alpha_k}}\norm{s_k}-\epsk\right)
= \Delta_k.
$$
This result together with~\eqref{Deltak-bd} completes the proof for this case.

\textbf{Case 2 ($\opt{} = \opttwo{}$).}
From $\xhat_{k+1} \in\Tepsk(x_k,\alpha_k)$ and the definition of $\epsilon_k$ in Step~\ref{line:triple2}, we have
$$
\phi(\xhat_{k+1};x_k,\alpha_k)
\leq \phi(T(x_k,\alpha_k);x_k,\alpha_k)
+ \gamma_2(
\phi(x_k;x_k,\alpha_k)
- \phi(T(x_k,\alpha_k);x_k,\alpha_k)).
$$
By adding the term $-\phi(x_k;x_k,\alpha_k)$ to both sides of the equation, rearranging terms, and then using the definition of $\epsk$ in Step~\ref{line:triple2}, we obtain
\begin{align}\label{ineq:diff-phip}
\phi(\xhat_{k+1};x_k,\alpha_k)
- \phi(x_k;x_k,\alpha_k)
&\leq (\gamma_2-1)(
\phi(x_k;x_k,\alpha_k)
- \phi(T(x_k,\alpha_k);x_k,\alpha_k)) \nonumber \\
&= \left(\tfrac{\gamma_2-1}{\gamma_2}\right)\epsk.
\end{align}
Using algebraic simplification and the definition of $\Delta_k$ in Step~\ref{line:Deltak2}, we obtain
\begin{align*}
\phi(x_k;x_k,\alpha_k)
- \phi(\xhat_{k+1};x_k,\alpha_k))
&= -\nabla f(x_k)^T s_k - \tfrac{1}{2\alpha_k}\|s_k\|^2 + r(x_k) - r(\xhat_{k+1}) \\
&= -\Delta_k - \tfrac{1}{2\alpha_k}\|s_k\|^2.
\end{align*}
By combining this equality with~\eqref{ineq:diff-phip} and recalling that $\gamma_2\in(0,1/2]$, we get
\begin{align}
-\Delta_k  
&= \phi(x_k;x_k,\alpha_k) - \phi(\xhat_{k+1};x_k,\alpha_k) + \tfrac{1}{2\alpha_k}\norm{s_k}^2 \nonumber\\
&\geq \left(\tfrac{1-\gamma_2}{\gamma_2}\right)\epsk + \tfrac{1}{2\alpha_k}\norm{s_k}^2
\geq \epsk + \tfrac{1}{2\alpha_k}\norm{s_k}^2. \label{eq:intermediate-2}
\end{align}
Applying Young's inequality to the product of norms in~\eqref{eq:itermediate} and using~\eqref{ineq:Terr}, we obtain
\begin{equation}\label{eq:itermediate-3}
\chipgk^2
\leq \tfrac{2\norm{\T(x_k,\alpha_k)-\hat x_{k+1}}^2}{\alpha_k^2} + \tfrac{2\norm{s_k}^2}{\alpha_k^2}
\leq \tfrac{4\epsk}{\alpha_k} + \tfrac{2\norm{s_k}^2}{\alpha_k^2}.
\end{equation}
Multiplying both sides of~\eqref{eq:itermediate-3} by $\alpha_k/4$, and then applying \eqref{eq:intermediate-2}, using the definition of $\Delta_k$ in  Step~\ref{line:Deltak2}, and Lemma~\ref{lem:epsilon-descent-full}(i) with $\xhat = \xhat_{k+1}$, $\xbar = x_k$, $s = s_k$, and $\epsilon = \epsk$ gives
\begin{align}
\tfrac{\alpha_k}{4}\chipgk^2
&\leq \epsk + \tfrac{1}{2\alpha_k}\|s_k\|^2
\leq -\Delta_k \nonumber \\
&= - \big(r(x_k+s_k) - r(x_k) + \nabla f(x_k)^T s_k \big)
\leq - D_{f+r}(x_k;s_k). \label{Deltak-bd-2}
\end{align}
Multiplying this inequality by a negative one completes the proof for this case.

It follows from $\alpha_k \in (0,\infty)$ and $\chipgk > 0$ that $-\beta\alpha_k \chipgk^2 < 0$, as claimed in~\eqref{def:betak}.
\end{proof}

It is convenient to define the following partition of iterations performed by Algorithm~\ref{alg:main} that depend on the status of the PG parameter:
\begin{align*}
\setpgSAME  &:= \{k\in \N{}: \text{ $\flagpgk = \samealpha$ in Line~\ref{line:update2}}\} \ \  \text{and} \\
\setpgDEC   &:= \{k\in \N{}: \text{ $\flagpgk = \decalpha$ in Line~\ref{line:update2b}}\}.
\end{align*}
Note that $\setpgSAME\cap\setpgDEC = \emptyset$ and that every iteration of the algorithm is in $\setpgSAME\cup\setpgDEC$.  Using these sets, the next result shows that the \textbf{while} loop in  Algorithm~\ref{alg:main} always terminates and that the new iterate produces a useful decrease in the objective function $f+r$. 

\begin{lemma}\label{lem:K-pg}
For each $k\in \N{}$, the \textbf{while} loop in Step~\ref{line:beginwhile} 
terminates finitely. Also, the following hold:
\begin{itemize}
\item[(i)] If $k\in\setpgSAME$, then $\alpha_{k+1} = \alpha_k$ and, with $\beta$ defined in~\eqref{def:betak}, it holds that
$$
f(x_{k+1}) + r(x_{k+1})
\leq f(x_k) + r(x_k) + \eta\Delta_k
\leq f(x_k) + r(x_k) - \eta\beta\alpha_k \chipgk^2.
$$
\item[(ii)] If $k\in\setpgDEC$, then $\alpha_{k+1} = \zeta\alpha_k$ and
$
f(x_{k+1}) + r(x_{k+1})
< f(x_k) + r(x_k).
$
\end{itemize}
Thus, the objective function $f+r$ is monotonically strictly decreasing over $\{x_k\}$.
\end{lemma}
\begin{proof}
From Lemma~\ref{lemma:per-iter-gain} we have $D_{f+r}(x_k;s_k) < 0$ so that standard results for a backtracking Armijo linesearch ensure that the \textbf{while} loop starting in Step~\ref{line:beginwhile} will terminate finitely, thus proving the first claim.

To prove part (i), suppose that $k\in\setpgSAME$. From the definition of $\setpgSAME$ we know that $\flagpgk = \samealpha$, and thus from Step~\ref{line:update2} that $\alpha_{k+1} = \alpha_k$ and that $j = 0$ when Step~\ref{line:update-pg-if} is reached.  This latter fact means that the inequality in Step~\ref{line:armijo-pg} does not hold for $j = 0$, meaning that the computed $x_{k+1}$ satisfies $x_{k+1} = \xtrial_{k+1,0} = x_k + s_k$ and $f(x_{k+1}) + r(x_{k+1})\leq f(x_k) + r(x_k)  +\eta\Delta_k$. Combining this inequality with Lemma~\ref{lemma:per-iter-gain} gives
$$
\label{dec-pg-1}
f(x_{k+1}) + r(x_{k+1}) 
\leq f(x_k) + r(x_k)  +\eta\Delta_k
\leq f(x_k) + r(x_k) - \eta\beta\alpha_k\chipgk^2
$$
as claimed, thus completing the proof of part (i).

To prove part (ii), suppose that $k\in\setpgDEC$.  From the definition of $\setpgDEC$ we know that $\flagpgk = \decalpha$, and therefore from Step~\ref{line:update2b} we have that $\alpha_{k+1} = \zeta\alpha_k$ and that $j > 0$ when Step~\ref{line:update-pg-if} is reached. Moreover, the linesearch in Steps~\ref{line:beginwhile}--\ref{line:endwhile} produces a vector $\xtrial_{k+1,j} \gets x_k + \xi^js_k$ that, with Lemma~\ref{lemma:per-iter-gain}, satisfies
\begin{equation}\label{dec-nsuff}
\begin{aligned}
f(\xtrial_{k+1,j}) + r(\xtrial_{k+1,j}) 
\leq f(x_k) + r(x_k) + \eta\xi^j\Delta_k
\leq f(x_k) + r(x_k) - \eta\xi^j \beta\alpha_k\chipgk^2.
\end{aligned}
\end{equation}
Combining this with $x_{k+1} \gets \xtrial_{k+1,j}$,
$\alpha_k > 0$, and $\chipgk > 0$ establishes that
$$
f(x_{k+1}) + r(x_{k+1})
= f(\xtrial_{k+1,j}) + r(\xtrial_{k+1,j})
\leq f(x_k) + r(x_k) -  \eta\xi^j\beta\alpha_k\chipgk^2
< f(x_k) + r(x_k),
$$
which completes the proof of part (ii).

Finally, the objective function $f+r$ is monotonically strictly decreasing over $\{x_k\}$ as a consequence of parts (i) and (ii), $\eta \in (0,1)$, $\beta > 0$, $\alpha_k > 0$, and $\chipgk > 0$.
\end{proof}

We now show that the PG parameters are bounded away from zero, thus implying that the unit step size is accepted for all iterations with sufficiently large index.

\begin{lemma} \label{lem:alpha-fixed}
For all $k\in\N{}$, the PG parameter $\alpha_k$ satisfies 
\begin{equation}\label{def:alphamin}
\alpha_k 
\geq \alpha_{\min} 
:= 
\begin{cases}
\min\big\{\alpha_0,\tfrac{3\gamma_1\zeta(1-\eta)}{L_g(1+\gamma_1)}\big\} & \text{if $\opt{} = \optone{}$,}\\
\min\big\{\alpha_0,\tfrac{\zeta(1-\eta)}{L_g}\big\} & \text{if $\opt{} = \opttwo{}$.}
\end{cases}
\end{equation}
Moreover, a bound on the number of times the PG parameter is decreased is given by
\begin{equation}\label{bound-Kalphadecr}
|\setpgDEC|
\leq 
c^\alpha_{\downarrow} 
:= \frac{\log(\alpha_{\min}/\alpha_0)}{\log(\zeta)}.
\end{equation}
\end{lemma}
\begin{proof}
We consider the two cases depending on the value of $\opt{}$.

\textbf{Case 1 ($\opt{}=\optone{}$).}
To prove~\eqref{def:alphamin}, we first establish the following:
\begin{equation}\label{to-show}
\text{if} \ 
\alpha_k\leq \tfrac{3\gamma_1(1-\eta)}{L_g(1+\gamma_1)}
\ \text{then} \
\eta \big(\tfrac{1}{\alpha_k}\norm{s_k}^2 - \sqrt{\tfrac{2\epsk}{\alpha_k}}\norm{s_k}-\epsk\big)
\leq (\tfrac{1}{\alpha_k}-\tfrac{L_g}{2})\norm{s_k}^2 - \sqrt{\tfrac{2\epsk}{\alpha_k}}\norm{s_k}- \epsk. 
\end{equation}
Using basic algebra, it may be shown that~\eqref{to-show} is equivalent to establishing that
\begin{equation}\label{to-show2}
\text{if} \ 
\alpha_k\leq \tfrac{3\gamma_1(1-\eta)}{L_g(1+\gamma_1)}
\ \text{then} \
\left(\tfrac{1}{\alpha_k} - \tfrac{L_g}{2(1-\eta)}\right)\norm{s_k}^2 \geq \sqrt{\tfrac{2\epsk}{\alpha_k}}\norm{s_k}+\epsk.
\end{equation}
Note that by rearranging \eqref{eq:lb-connection}, one can obtain 
\begin{equation}\label{eq:lb-connection-variation}
\tfrac{2-\gamma_1}{2(1+\gamma_1)\alpha_k}\norm{s_k}^2\geq \sqrt{\tfrac{2\epsk}{\alpha_k}}\norm{s_k}+\epsk.    
\end{equation}
Next, notice that
$
\text{if} \ 
\alpha_k\leq \tfrac{3\gamma_1(1-\eta)}{L_g(1+\gamma_1)}
\  \text{then}  \
\big(\tfrac{1}{\alpha_k} - \tfrac{L_g}{2(1-\eta)}\big)\norm{s_k}^2 \geq \tfrac{2-\gamma_1}{2(1+\gamma_1)\alpha_k}\norm{s_k}^2,
$
which together with \eqref{eq:lb-connection-variation} proves that~\eqref{to-show2} holds (equivalently, that~\eqref{to-show} holds).  Combining~\eqref{to-show} with $\xtrial_{k+1,0} = x_k+s_k = \xhat_{k+1}$, Lemma~\ref{lem:epsilon-descent-full}(ii) with $\xhat = \xhat_{k+1}$, $\xbar = x_k$, $\epsilon = \epsk$, and $\alphabar = \alpha_k$, and Step~\ref{line:Deltak1} yields
\begin{align*}
\text{if} \ 
\alpha_k\leq \tfrac{3\gamma_1(1-\eta)}{L_g(1+\gamma_1)}
\ \ \text{then} \ \
&f(\xtrial_{k+1,0}) + r(\xtrial_{k+1,0})  \\
&\leq f(x_k) + r(x_k) - \big(\tfrac{1}{\alpha_k}-\tfrac{L_g}{2}\big)\norm{s_k}^2 + \sqrt{\tfrac{2\epsk}{\alpha_k}}\norm{s_k}+ \epsk \\
&\leq  f(x_k) + r(x_k) - \eta \Big(\tfrac{1}{\alpha_k}\norm{s_k}^2 - \sqrt{\tfrac{2\epsk}{\alpha_k}}\norm{s_k}-\epsk\Big)\\
&= f(x_k) + r(x_k) + \eta \Delta_k.
\end{align*}
This inequality implies that the condition checked in Line~\ref{line:armijo-pg} for $j = 0$ will not hold so that $j = 0$ when Line~\ref{line:update-pg-if} is reached, meaning that the update $\alpha_{k+1} \gets \alpha_k$ will take place.  Summarizing, we have shown that, for any $k\in\N{}$ satisfying $\alpha_k \leq \tfrac{3\gamma_1(1-\eta)}{L_g(1+\gamma_1)}$, it holds that $\alpha_{k+1} = \alpha_k$.  Combining this observation with the fact that the $\alpha_0$ input to the algorithm can be any positive number, and that anytime the PG parameter is decreased it is by a factor $\zeta$ (see Step~\ref{line:update2b}), allows us to conclude that~\eqref{def:alphamin} holds.

\textbf{Case 2 ($\opt{}=\opttwo{}$).}
Note that the inequality in~\eqref{eq:intermediate-2} still holds, and thus
$-\Delta_k \geq \tfrac{1}{2\alpha_k}\norm{s_k}^2$ with $\Delta_k$ defined in Step~\ref{line:Deltak2}. Combining this inequality with the Lipschitz continuity assumption on $\nabla f$ (see Assumption~\ref{ass.first}),  $\xtrial_{k+1,0} = x_k + s_k$, and the definition of $\Delta_k$ in Step~\ref{line:Deltak2}, one obtains \begin{align}\label{eq:sufficient-decrease-opt2}
    &f(\xtrial_{k+1,0}) + r(\xtrial_{k+1,0}) \nonumber \\
    &\leq f(x_k) 
    + \grad f(x_k)^T s_k + \tfrac{L_g}{2}\norm{s_k}^2 + r(x_k+s_k)\nonumber \\
    &= f(x_k) + r(x_k) 
    + \grad f(x_k)^T s_k + \tfrac{L_g}{2}\norm{s_k}^2 + r(x_k+s_k) - r(x_k) \nonumber\\
    &= f(x_k) + r(x_k) + \Delta_k + \tfrac{L_g}{2}\norm{s_k}^2 
    \leq f(x_k) + r(x_k) +  \Delta_k(1-\alpha_kL_g).
\end{align}
From this inequality and $\Delta_k < 0$ (see Lemma~\ref{lemma:per-iter-gain}), we can observe that if $\alpha_k \leq (1-\eta)/L_g$, 
then $ f(\xtrial_{k+1,0}) + r(\xtrial_{k+1,0})\leq f(x_k) + r(x_k) + \eta\Delta_k$, which means that the condition checked in Line~\ref{line:armijo-pg} for $j = 0$ will not hold, which in turn 
means that $j = 0$ when Line~\ref{line:update-pg-if} is reached so that the update $\alpha_{k+1} \gets \alpha_k$ would take place. Combining this observation with the fact that the $\alpha_0$ input to the algorithm can be any positive number, and that anytime the PG parameter is decreased it is done so by a factor of $\zeta$ (see Step~\ref{line:update2b}), allows us to conclude that~\eqref{def:alphamin} holds.

As for~\eqref{bound-Kalphadecr}, an upper bound on $|\setpgDEC|$ is the smallest integer $\ell$ such that $\alpha_0\zeta^\ell \leq \alpha_{\min}$. Solving this inequality for $\ell$ shows that the result in~\eqref{bound-Kalphadecr} holds.
\end{proof}

The main theorem is now stated.  It gives an upper bound on the number of iterations performed by Algorithm~\ref{alg:main} before an approximate first-order solution is found.

\begin{theorem}\label{thm:complexity}
For any $\tau \in (0,1]$, the size of the set $\settau := \{k \in \N{} : \chipgk > \tau\}$ satisfies
\begin{equation}\label{final-set-bds}
\begin{aligned}
|\setpgSAME \cap \settau|
&\leq \cpg\tau^{-2} 
\ \ \ \text{with} \ \ \ 
\cpg := \frac{f(x_0)+r(x_0) - \inf_{x\in\R{n}} \big( f(x)+r(x)\big)}{\eta\beta\alpha_{\min}}
\end{aligned} 
\end{equation}
where $\beta$ is defined in~\eqref{def:betak} and $\alpha_{\min}$ is defined in Lemma~\ref{lem:alpha-fixed}. Moreover, combining this result with Lemma~\ref{lem:alpha-fixed} and the definition of  $c^\alpha_\downarrow$ in~\eqref{bound-Kalphadecr} shows that
\begin{equation} \label{final-complexity}
|\settau| \leq c^\alpha_\downarrow + \cpg\tau^{-2}.
\end{equation}
\end{theorem}
\begin{proof}
Let us define $\nu_k = f(x_k) + r(x_k) - \big( f(x_{k+1}) + r(x_{k+1})\big)$. This definition, Lemma~\ref{lem:K-pg},  Lemma~\ref{lem:alpha-fixed}, and the definition of $\settau$, it holds for arbitrary $\kbar \in \N{}$ that
\begin{equation*}
\begin{aligned}
&f(x_0) + r(x_0) - \big( f(x_{\kbar+1})+r(x_{\kbar+1}) \big)  \\
&= \sum_{0 \leq k \leq \kbar} \nu_k 
\geq \sum_{\substack{k\in\setpgSAME \cap \settau \\ 0 \leq k \leq \kbar}} \!\!\!\!\!\!\!\nu_k 
\geq \sum_{\substack{k\in\setpgSAME \cap \settau \\ 0 \leq k \leq \kbar}} \!\!\!\!\!\eta\beta\alpha_k\chipgk{^2} 
\geq \sum_{\substack{k\in\setpgSAME \cap \settau \\ 0 \leq k \leq \kbar}}\!\!\!\!\!\eta\beta\alpha_{\min}\tau^2.
\end{aligned}
\end{equation*}
Taking $\kbar\to\infty$ and using the monotonicity of the objective values in Lemma~\ref{lem:K-pg} gives
\begin{equation}\label{bd-Ksame}
f(x_0) + r(x_0) - \inf_{x\in\R{n}}  \big(f(x)+r(x) \big)
\geq |\setpgSAME \cap \settau|\eta\beta\alpha_{\min}\tau^2,
\end{equation}
which proves~\eqref{final-set-bds}. Finally, \eqref{final-set-bds},  $|\settau| = |\setpgSAME\cap \settau|+|\setpgDEC\cap \settau|$, and \eqref{bound-Kalphadecr} gives~\eqref{final-complexity}.
\end{proof}

Theorem~\ref{thm:complexity} shows that the worst-case iteration complexity for our \emph{inexact} PG method is $\Ocal(\tau^{-2})$, which is the same result that holds for the \emph{exact} PG method.  

\section{Sparse Regularizers and Finite Support Identification}\label{sec.identification}

In this section we focus our attention on the case when $r$ is chosen as the overlapping \grplone{} regularizer, whose associated proximal operator does not admit a closed-form solution.  This regularizer, which is studied in \cite{jenatton2011structured}, \cite{obozinski2011group}, and \cite{yuan2013efficient} is defined as
\begin{equation}\label{eq:natOverlap}
    r(x) = \sum_{i=1}^\ngrp [\lambda]_i\norm{[x]_{g_i}}
\end{equation}
where $\ngrp \in \N{}\setminus \{0\}$ denotes the number of groups, $\lambda\in\R{\ngrp}$ is a vector of strictly positive weights for the groups, and, for each $i\in[\ngrp]$, we use $g_i\subseteq[n]$ to denote the index set of the variables for the $i$th group and $[x]_{g_i}\in\R{|g_i|}$ to denote  the subvector of $x$ that corresponds to the elements in group $g_i$. Next, we denote the $j$th entry of the $i$th group by $g_{i,j}$ for every $i\in[\ngrp]$ and $j\in[|g_i|]$.  Finally, we assume that every variable is included in at least one group, i.e., we assume that $\cup_{i=1}^\ngrp g_i \equiv [n]$. This regularizer imposes structured sparsity on the set of solutions to  problem~\eqref{prob.main} (also see~\citep{obozinski2011group}).

Since this choice of $r$ is a special case of the more general regularizer considered in the previous section, we know that the worse-case iteration complexity of Theorem~\ref{thm:complexity} holds. Although this result is comforting, it does not shed light on whether the iterates generated by our inexact PG method can identify the support of an optimal solution.  In this section, we shall see that consideration of this topic is somewhat delicate. Indeed, no matter how accurately  problem~\eqref{def:pg-update} is approximately solved, it is not guaranteed that the approximate solution returned by Algorithm~\ref{alg:main} is sparse. Rather, one needs to exploit the geometric structure of $r(x)$ and design a specialized algorithm for approximately solving subproblem~\eqref{def:pg-update}.  

Let us formally define the support and the support identification.

\begin{definition}[support and support identification property]\label{def:id}
The \emph{support} of a point $x\in\R{n}$ with respect to $r$ in~\eqref{eq:natOverlap} is
$$
\Scal(x) := \{i \subseteq[\ngrp] ~|~ [x]_{g_i} \neq 0 \}.
$$
We say that an algorithm has the \emph{support identification property} if and only if it generates iterates $\{x_k\}$ satisfying $\Scal(x_k) = \Scal(x^*)$ for some solution $x^*\in\R{n}$ to~\eqref{prob.main} and all sufficiently large $k$.
\end{definition}

It is desirable that an algorithm has the support identification property. First, an algorithm with the support identification property is appropriate for identifying the most ``important" variables in a regression problem; an algorithm without such a property needs to perform ad-hoc post processing to obtain a sparse  estimate. Second, a solver that possesses  the support identification property is an appropriate choice to be used within second-order subspace acceleration methods (e.g., see~\cite{curtis2022}), which are known to be extremely efficient when the PG problem can be solved exactly.

Our method for approximately solving the PG subproblem, which will be proved to have the support identification property, is based on sufficiently reducing a certain primal-dual gap and using a special projection. We describe the dual formulation next.

\subsection{A dual formulation of the PG subproblem}

The PG subproblem~\eqref{def:pg-update} with regularizer $r$ given by~\eqref{eq:natOverlap} that is approximately solved during the $k$th iteration of Algorithm~\ref{alg:main} (see Line~\ref{line:triple1}, Line~\ref{line:triple2}, and~\eqref{def:ipg-update}) is given by
\begin{equation}\label{prob:primal-again}
x^*_k := \argmin{x\in\R{n}} \ \phi(x;x_k,\alpha_k)
\ \  \text{with}  \ \
\phi(x;x_k,\alpha_k) := \tfrac{1}{2\alpha_k}\norm{x-u_k}^2 + \sum_{i=1}^\ngrp [\lambda]_i\|[x]_{g_i}\|
\end{equation}
with $u_k = x_k - \alpha_k \nabla f(x_k)$.  
By comparing the definition of $x^*_k$ with~\eqref{def:pg-update}, we have 
\begin{equation}\label{def:xkstar}
x^*_k = \T(x_k,\alpha_k).
\end{equation}
Introducing a vector of auxiliary variables $z\in \R{\ngrp}$, subproblem~\eqref{prob:primal-again} is equivalent to
\begin{equation*}
\min_{x, z} \ 
 \tfrac{1}{2\alpha_k}\|x-u_k\|^2 +   \lambda^T z
 \ \ \ \text{s.t.} \ \ \
 \bbmatrix
 [x]_{g_i} \\ [z]_i
 \ebmatrix
 \in \Kcal_i
  \ \text{for all} \ i\in[\ngrp]
\end{equation*}
with 
$$
\Kcal_i
 := \left\{\bbmatrix v \\ \theta \ebmatrix 
~|~ v\in \R{|g_i|}, \theta\in\R{}, \ \text{and} \ \norm{v}\leq \theta\right\}
  \ \text{for all} \ i\in[\ngrp]
$$
so that $\Kcal_i$ is a second-order cone for each $i\in[\ngrp]$. 
Denoting the characteristic function by $\delta_{\Kcal_i}:\R{|g_i|+1}\to\{0,\infty\}$, which is defined as $\delta_{\Kcal_i}(w) = 0$ if $w\in \Kcal_i$ and is equal to $\infty$ otherwise, allows us to write this problem as 
\begin{equation}\label{prob:conic-primal-with-z}
\min_{x,z} 
\tfrac{1}{2\alpha_k}\|x-u_k\|^2 
+ \lambda^T z 
+ \sum_{i=1}^{\ngrp}  \delta_{\Kcal_i}\!\!
\left(
\bbmatrix
[x]_{g_i} \\ [z]_i
\ebmatrix
\right).
\end{equation}
Introducing a set of auxiliary vectors $\{p_i\}_{i=1}^\ngrp$  
with $p_i\in\R{|g_i|}$ for each $i\in[\ngrp]$ and an auxiliary vector $q\in\R{\ngrp}$, 
we may now rewrite~\eqref{prob:conic-primal-with-z} in the equivalent form
\begin{equation}\label{prob:conic-primal-separable}
\min_{x,z,\{p_i\},q} 
\tfrac{1}{2\alpha_k}\|x-u_k\|^2 
+ \lambda^T z 
+ \sum_{i=1}^\ngrp \delta_{\Kcal_i}
\left(
\bbmatrix p_i \\ [q]_i \ebmatrix
\right)
\ \ \text{s.t.} \ \ 
q = z \ \ \text{and} \ \
p_i = [x]_{g_i} \ \text{for all} \ i\in[\ngrp]. 
\end{equation}
If we define dual vectors $\pi\in\R{\ngrp}$ and $\{y_i\}_{i=1}^\ngrp$ with $y_i \in \R{|g_i|}$ for each $i\in [\ngrp]$, then the Lagrange function associated with the optimization problem~\eqref{prob:conic-primal-separable} takes the form 
\begin{align}
&\Lcal_k(x,z, \{p_i\},q, \pi, \{y_i\}) \label{lagrangian}\\
&:= \tfrac{1}{2\alpha_k}\|x-u_k\|^2 
+ \lambda^T z
+ \sum_{i=1}^\ngrp \delta_{\Kcal_i}
\left(
\bbmatrix p_i \\ [q]_i \ebmatrix
\right)
+ (q-z)^T\pi
+ \sum_{i=1}^\ngrp (p_i-[x]_{g_i})^T  y_i. \nonumber
\end{align}
The dual function is then given by
$$
\begin{aligned}
&\Dcal_k(\pi,\{y_i\}) 
:= \inf_{x,z,\{p_i\},q} \Lcal(x,z, \{p_i\},q,\pi,\{y_i\}) 
\\
&= \inf_{x,z}
\big\{
\tfrac{1}{2\alpha_k}\|x-u_k\|^2 
+ \lambda^Tz 
- \sum_{i=1}^\ngrp y_i^T [x]_{g_i} 
- z^T\pi 
\big\} 
- \sup_{\{p_i\},q}\ \sum_{i=1}^\ngrp \left(
-\bbmatrix y_i \\ [\pi]_i \ebmatrix^T
\bbmatrix p_i \\ [q]_i \ebmatrix
- \delta_{\Kcal_i}
\left(
\bbmatrix p_i \\ [q]_i \ebmatrix
\right)
\right)
\\
&= \inf_{x,z}
\big\{
\tfrac{1}{2\alpha_k}\|x-u_k\|^2 
+ \lambda^T z 
- \sum_{i=1}^\ngrp y_i^T [x]_{g_i} 
- z^T\pi 
\big\} 
- \sum_{i=1}^\ngrp \sup_{p_i,[q]_i}\ \left(
-\bbmatrix y_i \\ [\pi]_i \ebmatrix^T
\bbmatrix p_i \\ [q]_i \ebmatrix
- \delta_{\Kcal_i}
\left(
\bbmatrix p_i \\ [q]_i \ebmatrix
\right)
\right)
\\
&= \inf_{x,z}
\big\{
\tfrac{1}{2\alpha_k}\|x-u_k\|^2 
+ \lambda^T z 
- \sum_{i=1}^\ngrp y_i^T [x]_{g_i} 
- z^T\pi 
\big\} 
- \sum_{i=1}^\ngrp
\sup_{p_i, [q]_i}
\left\{ 
-\bbmatrix y_i \\ [\pi]_i \ebmatrix^T
\bbmatrix p_i \\ [q]_i \ebmatrix
~\middle|~
\bbmatrix p_i \\ [q]_i \ebmatrix\in\Kcal_i 
\right\}.
\end{aligned}
$$
If, for each $i\in[\ngrp]$, we define $\ybar_i\in\R{n}$ so that $[\ybar_i]_{g_i} = y_i$ and all other elements equal to zero, then it follows from the optimality conditions for the infimum problem above and the structure of the second-order cone $\Kcal_i$ appearing in the supremum that
\begin{align*}
\Dcal_k(\pi,\{y_i\})
&=
\begin{cases}
-\infty & \text{if $\pi\neq\lambda$ or  $\bbmatrix y_i \\ [\pi]_i \ebmatrix \notin\Kcal_i$ for any $i\in[\ngrp]$,} \\
-\frac{\alpha_k}{2} \norm{\sum_{i=1}^\ngrp \ybar_i}^2
- u_k^T (\sum_{i=1}^\ngrp \ybar_i) & \text{if $\pi=\lambda$ and $\bbmatrix y_i \\ [\pi]_i \ebmatrix \in\Kcal_i$ for all $i\in[\ngrp]$.}
\end{cases} 
\end{align*}
Thus, the dual problem obtained by maximizing the  function $\Dcal_k$ can be written as
\begin{equation}\label{dual-conic}
 \max_{\{y_i\}} \ 
 -\tfrac{\alpha_k}{2} 
 \Big\| \sum_{i=1}^\ngrp \ybar_i\Big\|^2
 - u_k^T \Big(\sum_{i=1}^\ngrp \ybar_i\Big) 
  \ \ \text{s.t.} \ \  
  \text{$\bbmatrix y_i \\ [\lambda]_i \ebmatrix \in \Kcal_i$
  for each $i\in[\ngrp]$.} 
\end{equation}
Next, note that if we define the matrix $A\in\R{n\times\sum_{i=1}^{\ngrp}|g_i|}$ such that
\begin{equation}\label{def:A}
[A]_{(g_{i,j}),j+\sum_{\ell=1}^{i-1}|g_\ell|} = 1 
\ \ 
\text{for all $i\in[\ngrp]$ and $j\in[|g_i|]$}
\end{equation}
and all other entries of $A$ are set equal to zero,
then it follows that
\begin{equation}\label{ident}
A\yhat = \sum_{i=1}^\ngrp \ybar_i
\ \ \text{where} \ \
\yhat :=
(y_1^T, y_2^T, \dots, y_{\ngrp}^T)^T\in\R{\sum_{i=1}^\ngrp |g_i|}.
\end{equation}
Introducing the set-valued mapping $\Mcal:\{1,2,\dots,\ngrp\} \rightrightarrows \{1,2,\dots,\sum_{i=1}^\ngrp |g_i|\}$ so that
\begin{equation}\label{def:M}
\Mcal(i)=
\left\{
\sum_{l=1}^{i-1}|g_l| + 1, 
\sum_{l=1}^{i-1}|g_l| + 2, \
\dots, \
\sum_{l=1}^{i-1}|g_l| + |g_i|
\right\},
\end{equation}
it follows that
$
y_i = 
[\yhat]_{\Mcal(i)}
$
for all $i\in[\ngrp]$. Using \eqref{ident} and~\eqref{dual-conic} yields the dual problem
\begin{equation}\label{prob:dual}
\Ycalhat(x_k,\alpha_k)
:= \Argmax{\yhat\in\Fcal_d} \ \phid(\yhat;x_k,\alpha_k) 
\ \ \text{with} \ \ 
\phid(\yhat;x_k,\alpha_k) 
:= -\tfrac{\alpha_k}{2} \| A\yhat\|^2- u_k^T A\yhat 
\end{equation}
where we recall that $u_k = x_k - \alpha_k \nabla f(x_k)$ and we define the dual feasibility set
\begin{equation}\label{K-lambda-ng}
\Fcal_d :=
\{\yhat\in\R{\sum_{i=1}^\ngrp |g_i|} ~|~ \|[\yhat]_{\Mcal(i)}\| \leq [\lambda]_i \text{ for each } i\in[\ngrp] \}.
\end{equation}
We use capital ``Argmax" in~\eqref{prob:dual} to emphasize that $\Ycalhat(x_k,\alpha_k)$ is a set. Note that strong duality holds for the primal problem~\eqref{prob:primal-again} and dual problem~\eqref{prob:dual} since Slater's condition holds for problem~\eqref{prob:dual} (since the components of $\lambda$ are all \emph{strictly} positive).   


We now establish results related to the dual solution set $\Ycalhat(x_k, \alpha_k)$ of problem~\eqref{prob:dual}. The first lemma establishes an important equation (the linking equation) that connects a dual solution to problem~\eqref{prob:dual} and the primal solution $x_k^*$ to problem~\eqref{prob:primal-again}.

\begin{lemma}\label{lemma:link}
The following results hold for the set of dual solutions.
\begin{itemize}
\item[(i)] The set of dual solutions $\Ycalhat(x_k,\alpha_k)$ for problem~\eqref{prob:dual} satisfies
\begin{equation}\label{eq:dual-prox-solset-form2}
\Ycalhat(x_k,\alpha_k)=\{\yhat\in\Fcal_d~|~A\yhat = (x^*_k - u_k)/\alpha_k\},
\end{equation}
where $A$ is defined in~\eqref{def:A} and $x_k^*$ is the primal solution defined in \eqref{def:xkstar}.
\item[(ii)] The solution $x^*_k$ to problem~\eqref{def:xkstar} satisfies  $x^{*}_k = u_k + \alpha_k A\yhat^{*}_k$ for all $\yhat^*_k \in\Ycalhat(x_k,\alpha_k)$. 
\item[(iii)] If $i\in[\ngrp]$ and $\yhat^*_k \in\Ycalhat(x_k,\alpha_k)$ satisfying $\norm{[\yhat^*_k]_{\Mcal(i)}} < [\lambda]_i$, then  $[x^*_k]_{g_i}=0$.
\end{itemize}
\end{lemma}
\begin{proof}
We begin with part (i). To prove~\eqref{eq:dual-prox-solset-form2} we first show that $\Ycalhat(x_k,\alpha_k) \subseteq \{\yhat\in\Fcal_d~|~A\yhat = (x^*_k - u_k)/\alpha_k\}$.  To this end, let $\yhat^*_k \in\Ycalhat(x_k,\alpha_k)$, and thus $\yhat^*_k\in\Fcal_d$ (see~\eqref{prob:dual}).  Since $x^*_k$ is the unique solution to the primal problem, we know that $(x^*_k, \yhat^*_k)$ is a primal-dual solution pair.  Therefore, it follows from first-order necessary optimality conditions obtained by setting the derivative (with respect to $x$) of the Lagrangian~\eqref{lagrangian} to zero that  $x^{*}_k = u_k + \alpha_k A\yhat^{*}_k$, where $A$ is defined in~\eqref{def:A}.  Since this equality and $\yhat^*_k\in\Fcal_d$ imply that $\yhat^*_k$ belongs to the set in the right-hand side of~\eqref{eq:dual-prox-solset-form2}, we have proved the first inclusion.  Now, to establish the inclusion  $\Ycalhat(x_k,\alpha_k) \supseteq \{\yhat\in\Fcal_d~|~A\yhat = (x^*_k - u_k)/\alpha_k\}$, let $\yhat^*_k$ satisfies  $\yhat^*_k\in\Fcal_d$ and $A\yhat^*_k = (x^*_k - u_k)/\alpha_k$.  Since $x^*_k$ is the unique solution to the primal problem and strong duality holds for the primal-dual pair, we know that there must exist $\yhat_{\text{sol}} \in \Ycalhat(x_k,\alpha_k)$.  Combining this with the inclusion that we already proved shows that $\yhat_{\text{sol}}$ must satisfy $A\yhat_{\text{sol}} = (x^*_k - u_k)/\alpha_k$. Combining this with the definition of the dual function $\phid(\cdot;x_k,\alpha_k)$ and  $A\yhat^*_k = (x^*_k - u_k)/\alpha_k$ shows that
\begin{align*}
\phid(\yhat^*_k;x_k,\alpha_k)
&= -\tfrac{\alpha_k}{2}\|A\yhat^*_k\|^2 - u_k^T A \yhat^*_k \\
&= -\tfrac{\alpha_k}{2}\|(x^*_k - u_k)/\alpha_k\|^2 - u_k^T (x^*_k - u_k)/\alpha_k  \\
&= -\tfrac{\alpha_k}{2}\|A\yhat_{\text{sol}}\|^2 - u_k^T A \yhat_{\text{sol}} 
= \phid(\yhat_{\text{sol}};x_k,\alpha_k).
\end{align*}
This equation and $\yhat^*_k\in\Fcal_d$ imply that $\yhat^*_k \in\Ycalhat(x_k,\alpha_k)$, thus completing the proof of this inclusion.  Since we have established both inclusions, we know that~\eqref{eq:dual-prox-solset-form2} holds. 

For (ii), let $\yhat^*_k \in\Ycalhat(x_k,\alpha_k)$.  Then, (i) implies  $A\yhat^*_k = (x^*_k - u_k)/\alpha_k$ so that (ii) holds.

For (iii), let $i\in[\ngrp]$ and $\yhat^*_k \in\Ycalhat(x_k,\alpha_k)$ satisfy $\norm{[\yhat^*_k]_{\Mcal(i)}} < [\lambda]_i$.  From $\yhat^*_k\in\Ycalhat(x_k,\alpha_k)$, first-order optimality conditions for~\eqref{prob:dual}, and  $\norm{[\yhat_k^*]_{\Mcal(i)}}<[\lambda]_i$ we have
\begin{equation}\label{eq:grad-zero}
\nabla_{[\hat y]_{\Mcal(i)}} \phid(\yhat_k^*;x_k,\alpha_k) = 0 
\ \ \Rightarrow \ \ \left[A^T(\alpha_k A\yhat_k^* + u_k)\right]_{\Mcal(i)}=0.
\end{equation}
The definition of $A$ gives $A_{g_i,\Mcal(i)} = I$ and $A_{g_i^c,\Mcal(i)} = 0$ , where $I$ is the identify matrix of size $|g_i|$ and $g_i^c$ is the complement  of $g_i$. Combining these facts with~\eqref{eq:grad-zero} gives
$$
0 
= \left[A^T(\alpha_k A\yhat_k^* + u_k)\right]_{\Mcal(i)} 
= [A^T]_{\Mcal(i),g_i}\left[\alpha_k A\yhat_k^* + u_k\right]_{g_i} 
= \left[\alpha_k A\yhat_k^* + u_k\right]_{g_i}.
$$
Combining this equation with $x^{*}_k = u_k + \alpha_k A\yhat^{*}_k$ implies that  $[x_k^*]_{g_i}=0$, as claimed.
\end{proof}

We now bound the distance between feasible points and the solution set $\Ycalhat(x_k, \alpha_k)$.

\begin{lemma}\label{leamma:feasbound}
The set $\Ycalhat(x_k, \alpha_k)$ is compact and convex, so that in particular the projection operator $\proj{\Ycalhat(x_k,\alpha_k)}{\cdot}$ and $\dist(\cdot,\Ycalhat(x_k,\alpha_k))$ are well defined.
Moreover, there exists $\nu_k>0$ and $\rho_k>0$ such that 
$$
\dist\big(\yhat,\Ycalhat(x_k,\alpha_k)\big)
\leq \nu_k \| A\yhat-A y \|^{\rho_k}
\ \ \text{for all} \ \ (\yhat,y)\in \big(\Fcal_d, \Ycalhat(x_k,\alpha_k) \big).
$$
\end{lemma}

\begin{proof}
From the definition of $\Ycalhat(x_k,\alpha_k)$ in~\eqref{eq:dual-prox-solset-form2} it is clear that $\Ycalhat(x_k,\alpha_k)$ is compact and convex, which proves the first part of the lemma. Next, by applying~\cite[Theorem 2.2]{luo1994error} with $X=\R{\sum_{i=1}^{\ngrp}|g_i|}$, $\bar X=\Fcal_d$, $f_i(\yhat)=\norm{[\yhat]_{\Mcal(i)}}^2-[\lambda]_i^2$ for all $i\in[\ngrp]$, $g(\yhat)=A\yhat-  (x^*_k - u_k)/\alpha_k$, it follows that there exists $\nu_k>0$ and $\rho_k>0$ such that 
$$
\dist(\yhat, S_k) \leq \nu_k\norm{A\yhat-(x^*_k - u_k)/\alpha_k}^{\rho_k} \text{ for all } \yhat\in \Fcal_d,
$$
where $S_k=\{\yhat\in\R{\sum_{i=1}^{\ngrp}|g_i|}~|~f_i(\yhat)\leq 0 \text{ for all } i\in[\ngrp] \text{ and } g(\yhat)=0\}$. Since $S_k \equiv \Ycalhat(x_k,\alpha_k)$ (see~\eqref{eq:dual-prox-solset-form2}), one can use the previous equality and  Lemma~\ref{lemma:link}(ii) to conclude 
\begin{align*}
\dist\big(\yhat, \Ycalhat(x_k,\alpha_k)\big) 
&\leq \nu_k\norm{A\yhat-(x^*_k - u_k)/\alpha_k}^{\rho_k} \\
&= \nu_k \|A\yhat-A\yhat^*_k\|^{\rho_k}
\ \ \text{for all} \ \ (\yhat,\yhat^*_k)\in \big(\Fcal_d,\Ycalhat(x_k,\alpha_k)\big),
\end{align*}
thus completing the proof.
\end{proof}

\subsection{A new algorithm for approximately solving the PG subproblem}

During the $k$th iteration of  Algorithm~\ref{alg:main}, it is necessary to compute $\xhat_{k+1}$ satisfying $\xhat_{k+1}\in\Tepsk(x_k,\alpha_k)$ for a particular value of $\epsk$. One way to perform  this task is to apply an iterative solver to the dual subproblem~\eqref{prob:dual} while monitoring the primal-dual gap.  For example, \eqref{prob:dual} can be approximately solved by the projected gradient-ascent algorithm with arc search~\cite[Section 2.3]{bertsekas1999nonlinear} armed with an appropriate early termination test.  Unfortunately, the straightforward application of such a method to the \emph{dual} problem is insufficient for discovering the support of an optimal solution to the \emph{primal} problem.  Therefore, in this section we present Algorithm~\ref{alg:dual-gradient-descent}, which is an enhanced projected gradient-ascent algorithm that we prove terminates with an $x_{k+1}\in\Tepsk(x_k,\alpha_k)$ and (asymptotically) to correctly identify the support of an optimal solution to problem~\eqref{prob.main} under a nondegeneracy assumption. 

Since we propose to use Algorithm~\ref{alg:dual-gradient-descent} as the subproblem solver to be used during the $k$th iteration of Algorithm~\ref{alg:main}, we denote its primal and dual iterate  sequences by $\{\xhat_{k,t},\yhat_{k,t}\}_{t\geq 0}$ where $t$ is the iteration counter of the subproblem solver. Given the $t$th dual iterate $\yhat_{k,t}$, motivated by Lemma~\ref{lemma:link} we compute a group index set $\Pcal_{k,t}$, a primal trial iterate $x_{k,t}$, and a projected primal trial iterate $\xhat_{k,t}$ in Line~\ref{line:pset} and Line~\ref{line:proj-primal}. The group index set $\Pcal_{k,t}$ holds the groups \emph{predicted} to be zero at a primal solution, and $\xhat_{k,t}$ is constructed by zeroing out all groups in $\Pcal_{k,t}$ so that $\xhat_{k,t}$ is at least as sparse as $x_{k,t}$. We draw your attention to the usage of $\epsilon_{k-1}^\iota$ appearing in the definition of $\Pcal_{k,t}$ in  Line~\ref{line:pset}, which is critical to ensuring that Algorithm~\ref{alg:dual-gradient-descent} is well posed and that our ultimate complexity result for support identification holds. Lines~\ref{line:inexact-start}--\ref{line:inexact-end} check for termination of the subproblem solver, the conditions of which are chosen to ensure that $\xhat_{k+1}\in \Tcal_{\epsk}(x_k, \alpha_k)$ anytime termination occurs (see~Lemma~\ref{lemma:well-pose-pgd} below and Lines~\ref{line:triple1} and~\ref{line:triple2} in Algorithm~\ref{alg:main}). If termination does not occur during the $t$th iteration, a standard projected gradient-ascent search is performed from Lines~\ref{line:pgd-search-start}--\ref{line:pgd-search-end} to compute $\yhat_{k,t+1}$. (There is nothing special about us using the projected gradient-ascent search; we simply need the sequence $\{\yhat_{k,t}\}_{t\geq 0}$ to converge to the solution set $\Ycalhat(x_k,\alpha_k)$.) 

\balgorithm[!th]
  \caption{An enhanced projected gradient-ascent method for solving problem~\eqref{prob:dual}.}
  \label{alg:dual-gradient-descent}
  \balgorithmic[1]
    \State \textbf{Input:} An initial dual solution estimate $\yhat_{k,0}\in\R{\sum_{i=1}^\ngrp |g_i|}$.  
    \State \textbf{Constants:} $\{\eta_2,\xi_2\}\subset (0,1)$ and $\{\sigma,\iota\} \subset (0,\infty)$
    \State \textbf{Values from Algorithm~\ref{alg:main}:} $x_k$, $\alpha_k$, $c_k$, $\epsilon_{k-1}$, $\opt{}$, and $\gamma_2$.
    \State Set $t \gets 0$, define $A$ as in~\eqref{def:A}, and set $u_k \gets x_k - \alpha_k \nabla f(x_k)$.
    \For{$t=1,2,\dots$}
    \State Set $\Pcal_{k,t} \gets \{i\in[\ngrp]~|~ \|[\yhat_{k,t}]_{\Mcal(i)}\| < [\lambda]_i - \epsilon_{k-1}^\iota\}$ with $\Mcal$ defined in~\eqref{def:M}.\label{line:pset}
    \State Define the trial primal iterate and trial  projected primal iterate as\label{line:proj-primal}
    $$
    x_{k,t} \gets u_k + \alpha_k  A\yhat_{k,t}  \ \ \text{and} \ \
    [\xhat_{k,t}]_{g_i} \gets
    \begin{cases}
    0, & \text{if $i\in\Pcal_{k,t}$}, \\
    [x_{k,t}]_{g_i}, & \text{if $i\notin\Pcal_{k,t}$,}
    \end{cases}
    \text{ for each } i\in[\ngrp].
    $$
    \If{ $(\opt{}=\optone{})$ }\label{line:inexact-start}
       \If{$\phi(\xhat_{k,t};x_k,\alpha_k) - \phid(\yhat_{k,t};x_k,\alpha_k) \leq  c_k\|\xhat_{k,t}-x_k\|^2$} \label{line:pgd-term11}
          \State \textbf{return} $\yhat_{k+1}=\yhat_{k, t}$ and $\xhat_{k+1}=\xhat_{k,t}$ 
        \EndIf
    \Else{ $(\opt{}=\opttwo{})$ }
        \If{$\phi(\xhat_{k,t};x_k,\alpha_k) - \phid(\yhat_{k,t};x_k,\alpha_k) \leq  \gamma_2(\phi(x_k;x_k,\alpha_k)-\phid(\yhat_{k,t};x_k,\alpha_k)$}\label{line:pgd-term21}
           \State \textbf{return} $\yhat_{k+1}=\yhat_{k, t}$ and $\xhat_{k+1}=\xhat_{k,t}$ 
        \EndIf      
    \EndIf   \label{line:inexact-end}
    \State $j \gets 0$.\label{line:pgd-search-start}
    \Loop 
    \State Set the trial step length $\sigma_{k,t} \gets \sigma\xi_2^j$.
    \State Set the trial iterate $\yhat_{k,t+1} \gets \proj{\Fcal_d}{\yhat_{k,t} + \sigma_{k,t}\nabla \phid(\yhat_{k,t};x_k,\alpha_k)}$. 
    \If{\!$\phid(\yhat_{k,t+1};x_k,\alpha_k)  
    \! \geq \! \phid(\yhat_{k,t};x_k,\alpha_k) 
     + \eta_2 \nabla \phid(\yhat_{k,t};x_k,\alpha_k)^T \!(\yhat_{k,t+1}  -  \yhat_{k,t})$\!} 
    \State break the inner loop
    \EndIf
    \State $j\gets j+1$.
    \EndLoop \label{line:pgd-search-end}
    \EndFor
  \ealgorithmic
\ealgorithm

We conclude this section by proving that  Algorithm~\ref{alg:dual-gradient-descent} is well posed.

\begin{lemma}\label{lemma:well-pose-pgd}
For each $k\in\N{}$, if $x_k$ does not satisfy the first-order optimality conditions for problem~\eqref{prob.main} with regularizer given by~\eqref{eq:natOverlap}, then  
Algorithm~\ref{alg:dual-gradient-descent} terminates finitely with $\xhat_{k+1}\in \Tcal_{\epsilon_k}(x_k, \alpha_k)$ with $\epsilon_k$ defined in Line~\ref{line:triple1} of Algorithm~\ref{alg:main} if \opt{}=\optone{} or with $\epsilon_k$ defined in Line~\ref{line:triple2} of Algorithm~\ref{alg:main} if \opt{}=\opttwo{}.
\end{lemma}
\begin{proof}
For a proof by contradiction, assume that Algorithm~\ref{alg:dual-gradient-descent} never terminates, and thus it generates an infinite subsequence of iterates $\{\yhat_{k,t}\}_{t\geq 0}$ that is equivalent to the sequence generated by the projected gradient-ascent with arc search algorithm by~\cite[Section 2.3]{bertsekas1999nonlinear}. It is well known (e.g., see~\cite{iusem2003convergence}) that such an algorithm when applied to a concave optimization problem will satisfy 
$\lim_{t\to\infty} \phid(\yhat_{k,t};x_k,\alpha_k) = \phid(\yhat^*_k;x_k,\alpha_k)$, for any $\yhat^*_k$ that is an optimal solution to the dual problem~\eqref{prob:dual}. Combining this with the linking equation (see Lemma~\ref{lemma:link}(iii)) and Line~\ref{line:proj-primal} of Algorithm~\ref{alg:dual-gradient-descent}, it follows that $\lim_{t\to\infty} \phi(x_{k,t};x_k,\alpha_k) =  \phi(x_k^*;x_k,\alpha_k)$ where $\xhat^*_k$ is the optimal solution to the primal problem~\eqref{prob:primal-again}.  Therefore, since strong duality holds for our primal-dual problems (Slater's condition holds for problem~\eqref{prob:dual}), we can conclude that
\begin{equation}\label{term:lhs}
\lim_{t\to\infty} \big(\phi(x_{k,t};x_k,\alpha_k)-\phid(\yhat_{k,t};x_k,\alpha_k) \big) = 0.
\end{equation}

Next, we claim that there exists a subsequence $\Tcal$ such that $\lim_{t\in\Tcal} \xhat_{k,t} = x_k^*$.  For a proof by contradiction, suppose that no such subsequence $\Tcal$ exists. Combining this supposition with $\lim_{t\to\infty} x_{k,t} = x_k^*$ (this follows from~\eqref{term:lhs} and the fact that the primal proximal subproblem has a unique solution), we can conclude that there exists an $i_*\in[\ngrp]$, $t_*\in\N{}$, 
 and constants $\epsilon$ and $\Gcal$ such that
\begin{equation}\label{G-bound}
\|[\xhat_{k,t} - x_{k,t}]_{g_{i_*}}\| 
\geq \epsilon >  0 
\ \ 
\text{and} \ \
\|[x_{k,t} - x_k^*]_{g_{i_*}}\| < \epsilon.
\ \ \text{for all $t\geq t_*$.}
\end{equation}
The first inequality in~\eqref{G-bound} together with the definition of $\xhat_{k,t}$ in Line~\ref{line:proj-primal} shows that $i_* \in\Pcal_{k,t}$ for all $t\geq t_*$.  Using this fact, the definition of $\Pcal_{k,t}$, and~\eqref{G-bound} gives
\begin{equation}\label{in-Pcal}
i_*\in\Pcal_{k,t}
\ \ \text{and} \ \
\|[\yhat_{k,t}]_{\Mcal(i_*)}\| < [\lambda]_{i_*} - \epsilon_{k-1}^\iota
\ \ \text{for all $t\geq t_*$.}
\end{equation}
Using $\lim_{t\to\infty} \dist\big(\yhat_{k,t},\Ycalhat(x_k,\alpha_k)\big) = 0$, $\epsilon_{k-1} > 0$, and~\eqref{in-Pcal}, we know that there exists a $y_k^*\in\Ycalhat(x_k,\alpha_k)$ satisfying $\|[y_k^*]_{\Mcal(i_*)}\| < [\lambda]_{i_*}$. This fact and Lemma~\ref{lemma:link}(iii) it follows that $[x_k^*]_{g_{i_*}} = 0$.  Using this equality with the triangle inequality, $i_*\in\Pcal_{k,t}$ for all $t\geq t_*$ (see~\eqref{in-Pcal}) meaning that $[\xhat_{k,t}]_{g_{i_*}} = [x_k^*]_{g_{i_*}}$, and the second inequality in~\eqref{G-bound} yields
$$
\|[\xhat_{k,t} - x_{k,t}]_{g_{i_*}}\|
\leq \|[\xhat_{k,t} - x_k^*]_{g_{i_*}} \| + \|[x_k^* - x_{k,t}]_{g_{i_*}}\|
= \|[x_k^* - x_{k,t}]_{g_{i_*}}\|
< \epsilon
\ \ \text{for all $t\geq t_*$,}
$$
which contradicts the first inequality in~\eqref{G-bound}.  Therefore, our claim must be true, namely that there exists a subsequence $\Tcal$ such that $\lim_{t\in\Tcal} \xhat_{k,t} = x_k^*$. 

Using $\lim_{t\in\Tcal} \xhat_{k,t} = x_k^*$, we find that
\begin{equation}\label{term:lhs-old}
\lim_{t\in\Tcal} \big(\phi(\xhat_{k,t};x_k,\alpha_k)-\phid(\yhat_{k,t+1};x_k,\alpha_k) \big) = 0.
\end{equation}
Moreover, since by assumption $x_k$ does not satisfy the first-order optimality conditions for problem~\eqref{prob.main}, we know that there exists a constant $\delta > 0$ such that
$
\min\{c_k\|\xhat_{k,t}-x_k\|^2, \gamma_2(\phi(x_k;x_k,\alpha_k) - \phid(\yhat_{k,t+1};x_k,\alpha_k))\} \geq \delta > 0
\ \ \text{for all sufficiently large $t \in\Tcal$.}
$
This fact may be combined with~\eqref{term:lhs-old} to conclude that the conditions in Lines~\ref{line:pgd-term11} and~\ref{line:pgd-term21} will both hold for all sufficiently large $t\in\Tcal$, therefore proving that Algorithm~\ref{alg:dual-gradient-descent} will finitely terminate (regardless of the value of $\opt{}$), which contradicts our original supposition, and thus Algorithm~\ref{alg:dual-gradient-descent} finitely terminates.

Next, suppose that $\opt{}=\optone{}$.  In this case, Algorithm~\ref{alg:dual-gradient-descent} will finitely terminate in Line~\ref{line:pgd-term11} and return the vector $\xhat_{k+1}$ satisfying 
$$
\phi(\xhat_{k+1};x_k,\alpha_k) - \phi(T(x_k, \alpha_k);x_k,\alpha_k) \leq \phi(\xhat_{k+1};x_k,\alpha_k) - \phid(\yhat_{k+1};x_k,\alpha_k) \leq c_k\norm{\xhat_{k+1}-x_k}^2.
$$
This inequality and~\eqref{def:ipg-update} show that $\xhat_{k+1}\in \Tcal_{\epsilon_k}(x_k, \alpha_k)$ with $\epsilon_k=c_k\norm{\xhat_{k+1}-x_k}^2$ as defined in Line~\ref{line:triple1} of Algorithm~\ref{alg:main}, which completes the proof for this case. 

Finally, suppose that $\opt{}=\opttwo{}$.  In this case, Algorithm~\ref{alg:dual-gradient-descent} will finitely terminate in Line~\ref{line:pgd-term21} and return the vector $\xhat_{k+1}$ satisfying
$$
\phi(\xhat_{k+1};x_k,\alpha_k) - \phid(\yhat_{k+1};x_k,\alpha_k) \leq \gamma_2(\phi(x_k;x_k,\alpha_k) - \phid(\yhat_{k+1};x_k,\alpha_k)).
$$
Rearranging terms and using $\gamma_2\in(0,1/2]$ and weak duality gives
$$
\phi(\xhat_{k+1};x_k,\alpha_k) - \gamma_2 \phi(x_k;x_k,\alpha_k)
\leq (1-\gamma_2)\phid(\yhat_{k+1};x_k,\alpha_k)
\leq (1-\gamma_2) \phi(\T(x_k,\alpha_k);x_k,\alpha_k).
$$
Thus, $\phi(\xhat_{k+1};x_k,\alpha_k) - \phi(\T(x_k,\alpha_k);x_k,\alpha_k) 
\leq  \gamma_2( \phi(x_k;x_k,\alpha_k) - \phi(\T(x_k,\alpha_k);x_k,\alpha_k))$, so that  $\xhat_{k+1}\in \Tcal_{\epsilon_k}(x_k, \alpha_k)$ with $\epsilon_{k}=\gamma_2( \phi(x_k;x_k,\alpha_k) -  \phi(\T(x_k,\alpha_k);x_k,\alpha_k)$ as defined in Line~\ref{line:triple2} of Algorithm~\ref{alg:main}, which completes the proof for this case.
\end{proof}

\subsection{Support Identification}

In this section, we prove that Algorithm~\ref{alg:main}, when  paired with Algorithm~\ref{alg:dual-gradient-descent} as its subproblem solver, has a finite support identification property.  This result is proved under the following two assumptions. 

\begin{assumption}\label{ass:active-set}
The function $f$ is strongly convex with strong convexity parameter $\mu_f > 0$. Thus, problem~\eqref{prob.main} with regularizer defined by~\eqref{eq:natOverlap} has a unique minimizer $x^*$. 
\end{assumption}

\begin{assumption}\label{ass:alpha}
The parameter $\gamma_1\in(0,2)$ from Algorithm~\ref{alg:main} is  chosen to satisfy $3\gamma_1/(1+\gamma_1) \leq 1$, and the initial PG parameter is then chosen to satisfy 
$$
\alpha_0 
\in
\begin{cases}
\left(0,\tfrac{3\gamma_1(1-\eta)}{L_g(1+\gamma_1)}\right] & \text{if $\opt{} = \optone{}$,} \\[0.5em]
\left(0,\tfrac{1-\eta}{L_g}\right] & \text{if $\opt{} = \opttwo{}$}.
\end{cases}
$$
\end{assumption}

Assumption~\ref{ass:alpha} is made to simplify the presentation, since it ensures that the sequence $\{\alpha_k\}$ is constant and smaller than $1/L_g$ (regardless of the value of $\opt{}$) as shown in Lemma~\ref{lem-epsk-to-zero}. Otherwise, the analysis below holds for all iterations beyond some fixed iteration, as a consequence of Lemma~\ref{lem:alpha-fixed}.

The following limits hold as a consequence of Assumption~\ref{ass:active-set} and Assumption~\ref{ass:alpha}.

\begin{lemma}\label{lem-epsk-to-zero}
Under Assumption~\ref{ass:active-set} and Assumption~\ref{ass:alpha}, with $\alpha^* := \alpha_0$, it holds that
\begin{align*}
\lim_{k\to\infty} \epsk &= 0, \ \ 
\lim_{k\to\infty} \|s_k\| = 0, \ \ 
\lim_{k\to\infty} x_k = x^*, \ \
\lim_{k\to\infty} x^*_k = x^*, \ \
\text{and} \\ 
\alpha_k &= \alpha^* = \alpha_0\in(0,(1-\eta)/L_g] \  \text{for all $k\in\N{}$.}
\end{align*}
\end{lemma}
\begin{proof}
The fact that $\alpha_k = \alpha^* = \alpha_0$ for all $k\in\N{}$ follows from Assumption~\ref{ass:alpha}, the definition of $\alpha^*$, and the proof of Lemma~\ref{lem:alpha-fixed}. The inclusion $\alpha_0 \in (0,(1-\eta)/L_g]$ follows from Assumption~\ref{ass:alpha}.  Next, 
Lemma~\ref{lem:K-pg}(i) and Lemma~\ref{lem:alpha-fixed} imply that $f(x_{k+1}) + r(x_{k+1})
\leq f(x_k) + r(x_k) + \eta \Delta_k
\leq f(x_k) + r(x_k) - \eta\beta\alpha^* \chipgk^2$ for all $k\in\N{}$.  These inequalities and the fact that $\{f(x_k)+r(x_k)\}$ is bounded below because of Assumption~\ref{ass:active-set}, show that $\{\chipgk\}\to 0$ and $\{\Delta_k\} \to 0$.  Using these limits and Lemma~\ref{lem:alpha-fixed}, it follows from~\eqref{Deltak-bd} when $\opt{}=\optone{}$ and from~\eqref{Deltak-bd-2} when $\opt{}=\opttwo{}$ that 0 = $\lim_{k\to\infty} \epsk = \lim_{k\to\infty} \|s_k\|$, which are the first two limits we aimed to prove. Next, it follows from $\{\chipgk\} \to 0$, the fact that $\chipgk$ is a first-order optimality measure for~\eqref{prob.main} (see~\cite[Theorem~10.7]{beck2017first}), and Assumption~\ref{ass:active-set} that $\lim_{k\to\infty} x_k = x^*$. 
Next, combining $\{\epsk\} \to 0$,  Lemma~\ref{lemma:chibound}, and Lemma~\ref{lem:alpha-fixed} proves that $\{\|x^*_k-x_k\|\} \to 0$.  Combining this limit with $\{x_k\}\to x^*$ and $\|x^*_k - x^*\| \leq \|x^*_k-x_k\| + \|x_k-x^*\|$ gives $\lim_{k\to\infty} x^*_k = x^*$. 
\end{proof}

We now show that the dual solutions to problem~\eqref{prob.main} are given by $\Ycalhat(x^*,\alpha^*)$.

\begin{lemma}\label{lem:dual-at-star}
Under Assumption~\ref{ass:active-set}, the solution set to the dual of problem~\eqref{prob.main} is 
\begin{align}\label{eq:prob.main.dual}
    \Ycalmain := \max_{\yhat\in \Fcal_d} \ \phidmain(\yhat) 
    \ \ \text{with} \ \ 
    \phidmain(\yhat) := \inf_{x\in\R{n}} \big( f(x)-x^TA\yhat \big).
\end{align}
Also, if Assumption~\ref{ass:alpha} holds, then the dual solutions satisfy $\Ycalmain \equiv \Ycalhat(x^*,\alpha^*)$. 
\end{lemma}
\begin{proof}
Following a similar derivation as that which led to the dual~\eqref{prob:dual} for the PG subproblem~\eqref{prob:primal-again}, one can see that the dual problem for~\eqref{prob.main} can be written as~\eqref{eq:prob.main.dual}.

Before proving $\Ycalmain \equiv \Ycalhat(x^*,\alpha^*)$, we consider a shifted version of the dual~\eqref{prob:dual} with $(x_k,\alpha_k) = (x^*,\alpha^*)$ and $u^* = x^* - \alpha^* \nabla f(x^*)$ where $\alpha^*$ is defined in Lemma~\ref{lem-epsk-to-zero}:
\begin{align}\label{eq:prob.prox.dual.shifted}
    \max_{\yhat\in\Fcal_d} \ \phidprox(\yhat)
\end{align}
where
\begin{align*}
\phidprox(\yhat)
&:= -\tfrac{\alpha^*}{2}\|A\yhat\|^2 - (u^*)^TA\yhat + f(x^*) - \tfrac{\alpha^*}{2}\|\nabla f(x^*)\|^2  \\
&\phantom{:}= \inf_{x\in\R{n}} 
    \big( \tfrac{1}{2\alpha^*}\|x-u^*\|^2 - x^TA\yhat + f(x^*) - \tfrac{\alpha^*}{2}\norm{\nabla f(x^*)}^2  \big).
\end{align*}
This shifted dual problem serves as a bridge between  problems~\eqref{eq:prob.main.dual} and~\eqref{prob:dual}, and has two important properties. First,
$\Ycalhat(x^*,\alpha^*)$ is the solution set of problem~\eqref{eq:prob.prox.dual.shifted} since it is also the solution set to problem~\eqref{prob:dual} (see Lemma~\ref{lemma:link}(i)). Second, the following hold:
\begin{equation}\label{phi-eq-shift}
\phidprox(\pi) =  f(x^*) - (x^*)^T\nabla f(x^*)
\ \ \text{and} \ \ 
A\pi = \nabla f(x^*)
\ \ \text{for all $\pi\in\Ycalhat(x^*,\alpha^*)$,}
\end{equation}
which can be seen to hold from the following argument. Let $\pi\in\Ycalhat(x^*,\alpha^*)$. Since $x^*$ is a solution to problem~\eqref{prob.main}, it is easy to show that $x^*$ is also a solution to problem~\eqref{prob:primal-again}.  Therefore, $(x^*, \pi)$ is a primal-dual solution pair for the proximal problem and its dual with $(x_k,\alpha_k) = (x^*,\alpha^*)$.  We can now apply Lemma~\ref{lemma:link}(ii) to conclude that $x^*= u^* + \alpha^* A \pi = x^* - \alpha^*\nabla f(x^*) + \alpha^* A \pi$, which implies the second equality in~\eqref{phi-eq-shift}.  It now follows from the second equality in~\eqref{phi-eq-shift} and
the definition of $u^*$ that 
$\phidprox(\pi)
= -\tfrac{\alpha^*}{2}\|A\pi\|^2 - (u^*)^TA\pi + f(x^*) - \tfrac{\alpha^*}{2}\|\nabla f(x^*)\|^2 
= -\alpha^*\|\nabla f(x^*)\|^2 - (x^*-\alpha^*\nabla f(x^*))^T\nabla f(x^*) + f(x^*) 
= f(x^*) - (x^*)^T\nabla f(x^*)$, which establishes the first equality in~\eqref{phi-eq-shift}.

We first prove that $\Ycalhat(x^*,\alpha^*) \subseteq \Ycalmain$.  To that end, suppose that $\yhat^*\in\Ycalhat(x^*,\alpha^*)$ and define (see the function inside the infimum in the definition of $\phidprox$) the vector
$$
\xprox := \argmin{x\in\R{n}} \ \tfrac{1}{2\alpha^*}\|x-u^*\|^2 
- x^TA\yhat^*
+ f(x^*) - \tfrac{\alpha^*}{2}\norm{\nabla f(x^*)}^2.
$$
First-order optimality conditions for this problem show that $\frac{1}{\alpha^*}(\xprox - x^*) + \grad f(x^*) = A\yhat^*,
$
which combined with $A\yhat^* = \nabla f(x^*)$ (see~\eqref{phi-eq-shift}) proves that $\xprox=x^*$. It is also straightforward to verify using  first-order optimality conditions and $A\yhat^* = \nabla f(x^*)$ that  $x^* = \argmin{x\in\R{n}} \ \big( f(x)-x^TA\yhat^* \big)$, so that
\begin{equation}\label{phi-eq-main}
\phidmain(\yhat^*)
= f(x^*) - (x^*)^T A\yhat^*
= f(x^*) - (x^*)^T \nabla f(x^*).
\end{equation}

Next, it follows from the definition of $u^*$, Assumption~\ref{ass.first}, $\alpha^*\in (0,(1-\eta)/L_g]$ (see Lemma~\ref{lem-epsk-to-zero}), and $\eta\in(0,1)$  that
\begin{align*}
&\tfrac{1}{2\alpha^*}\|x-u^*\|^2 + f(x^*) - \tfrac{\alpha^*}{2}\norm{\nabla f(x^*)}^2  \\
&= f(x^*) + (x-x^*)^T\grad f(x^*) + \tfrac{1}{2\alpha^*}\norm{x-x^*}^2\geq f(x) \text{ for all } x\in\R{n}.
\end{align*}
Substracting $x^TA\yhat$ from both sides of the previous inequality and then taking the infimum over $x\in\R{n}$, we find that $\phidprox(\yhat) \geq \phidmain(\yhat)$ for all $\yhat\in \Fcal_d$.  Using this inequality, \eqref{phi-eq-main}, \eqref{phi-eq-shift}, and  $\yhat^*\in\Ycalhat(x^*,\alpha^*)$ (recall that we earlier commented that $\Ycalhat(x^*,\alpha^*)$ is also the solution set to problem~\eqref{eq:prob.prox.dual.shifted}) that
$$
\phidmain(\yhat^*) = \phidprox(\yhat^*) = f(x^*)-(x^*)^T\grad f(x^*) \geq \phidprox(\yhat) \geq \phidmain(\yhat) \text{ for all } \yhat\in  \Fcal_d,
$$
which shows $\yhat^*\in \Ycalmain$, and completes the proof for this inclusion.

Next,  we prove the inclusion  $\Ycalmain\subseteq\Ycalhat(x^*,\alpha^*)$. Let $\yhat^*\in\Ycalmain$. If we denote $\xmain :=  \text{argmin}_{x\in\R{n}}\ ( f(x) - x^TA\yhat^* )$, then first-order optimality conditions give
\begin{equation} \label{eq:Ay-2}
\grad f(\xmain)=A\yhat^*.
\end{equation}
Since $(x^*,\yhat^*)$ is an  optimal primal-dual solution pair for problem~\eqref{prob.main} and its dual~\eqref{eq:prob.main.dual}, it follows that $A\yhat^*=\grad f(x^*)$.  Together with~\eqref{eq:Ay-2} shows that $\grad f(x^*)=\grad f(\xmain)$, which combined with $f$ being strongly convex with parameter $\mu_f > 0$ (see Assumption~\ref{ass:active-set}) to obtain $0 = \norm{\grad f(x^*)-\grad f(\xmain)} \geq \mu\norm{x^*-\xmain}$, so that  $\xmain=x^*$. 
Using the definition of $\xmain$, $\xmain = x^*$, and $A\yhat^*=\grad f(x^*)$, we have that 
\begin{equation}\label{eq:phidmainval}
    \phidmain(\yhat^*)= f(\xmain) - (\xmain)^TA\yhat^*=f(x^*) -(x^*)^T\grad f(x^*).
\end{equation}
Next, using $A \yhat^* = \nabla f(x^*)$ and the definition of $u^*$, it follows that
\begin{equation}\label{eq:phidproxval}
\begin{aligned}
\phidprox(\yhat^*)
&= -\tfrac{\alpha^*}{2}\|A\yhat^*\|^2 - (u^*)^TA\yhat^* + f(x^*) - \tfrac{\alpha^*}{2}\|\nabla f(x^*)\|^2 \\
&= -\alpha^*\|\nabla f(x^*)\|^2 - (x^*-\alpha^*\nabla f(x^*))^T\nabla f(x^*) + f(x^*) \\
&= f(x^*) - (x^*)^T\nabla f(x^*).
\end{aligned}    
\end{equation}
Taking any $\yhat'\in \Ycalhat(x^*,\alpha^*)$, it follows from~\eqref{phi-eq-shift} that $\phidprox(\yhat')=f(x^*) -(x^*)^T\grad f(x^*)$, which together with~\eqref{eq:phidmainval} and~\eqref{eq:phidproxval} implies
$$
\phidprox(\yhat^*) = \phidmain(\yhat^*) = \phidprox(\yhat') \geq \phidprox(\yhat) \  \text{for all} \  \yhat \in \Fcal_d.
$$
Since this inequality means $\yhat^*\in\Ycalhat(x^*,\alpha^*)$, we conclude that  $\Ycalmain\subseteq\Ycalhat(x^*,\alpha^*)$.
\end{proof}

Since we established in Lemma~\ref{lem:dual-at-star} that $\Ycalhat(x^*,\alpha^*)$ is the set of dual solutions to problem~\eqref{prob.main}, we may now present our non-degeneracy assumption, which uses the support of $x^*$.  Note that each $i\in[\ngrp]$ and $\yhat\in\Ycalhat(x^*,\alpha^*)$ together satisfy $\| [\yhat]_{\Mcal(i)}\| \leq [\lambda]_i$ by the definition of $\Ycalhat(x^*,\alpha^*)$.  Therefore, the following non-degeneracy assumption is a strengthening of this inequality for all groups not in $\Scal(x^*)$.

\begin{assumption}\label{ass:nd}
The quantity
$$
\delta_{\text{nd}}
:= 
\begin{cases}
\min_{\yhat\in\Ycalhat(x^*,\alpha^*),  i\not\in\Scal(x^*)}     \big([\lambda]_i-\| [\yhat]_{\Mcal(i)} \| \big) & \text{if $\Scal(x^*) \subsetneqq [\ngrp]$,} \\
1 & \text{if $\Scal(x^*) = [\ngrp]$,}
\end{cases}
$$
satisfies $\delta_{\text{nd}} > 0$.  It follows that $\delta^* := \min\{1,\delta_{\text{nd}}\} \in (0,1]$.
\end{assumption}

The essence of the previous assumption is that, for each $i\notin \Scal(x^*)$, the set $\{[\yhat]_{\Mcal(i)}\}_{\yhat\in\Ycalhat(x^*,\alpha^*)}$ is bounded away from the boundary of the two-norm ball centered at zero of radius $[\lambda]_i$.  This  non-degeneracy assumption is crucial for  proving support identification results, where we note that a similar assumption has been used in \cite{sun19a,curtis2022}. In fact, Assumption~\ref{ass:nd} is an extension of the assumption used by these authors that is applicable to the overlapping $\grplone$ case. 

We also need minimal control over the sequences $\{\nu_k\}$ and $\{\rho_k\}$ from Lemma~\ref{leamma:feasbound}. 

\begin{assumption}\label{ass:upbd}
The sequences $\{\nu_k\}$ and $\{\rho_k\}$ defined in Lemma~\ref{leamma:feasbound} can be chosen to satisfy $\sup_{k \geq 0} \nu_k =: \nu_{\max} < \infty$ and $\inf_{k \geq 0} \rho_k =: \rho_{\min} > 0$.
\end{assumption}

Let us also choose any $\nu^*\in(0,\infty)$ and $\rho^* \in (0,\infty)$ satisfying
\begin{equation}\label{dist-star}
\dist\big(\yhat,\Ycalhat(x^*,\alpha^*)\big)
\leq \nu^* \| A\yhat-A y \|^{\rho^*}
\ \ \text{for all} \ \ (\yhat,y)\in \big(\Fcal_d, \Ycalhat(x^*,\alpha^*) \big), 
\end{equation}
which are guaranteed to exist because of Lemma~\ref{leamma:feasbound} with $(x_k,\alpha_k) = (x^*,\alpha^*)$. Using these quantities, we may now bound the distance between any \emph{exact} solution of the dual problem \eqref{prob:dual} and $\Ycalhat(x^*,\alpha^*)$, and the distance between the \emph{inexact} solution of \eqref{prob:dual} computed by Algorithm~\ref{alg:main} and the set $\Ycalhat(x_k,\alpha_k)$.

\begin{lemma}\label{lemma:two-bounds}
Under Assumptions~\ref{ass:active-set}--\ref{ass:upbd} the following hold. 
\begin{itemize}
    \item[(i)] With $\nu^*$ and $\rho^*$ from~\eqref{dist-star} and $\alpha^*$ from Lemma~\ref{lem-epsk-to-zero}, we have for each $k\in\N{}$ that 
        $$
         \dist(\yhat_k^*,\Ycalhat(x^*,\alpha^*)) \leq \nu^*\left(\tfrac{\norm{x_k^*-x^*}}{\alpha_k} + \left(L_g+\tfrac{1}{\alpha_k}\right)\norm{x_k-x^*}\right)^{\rho^*}
         \ \text{for all $\yhat_k^*\in\Ycalhat(x_k,\alpha_k)$.}
        $$
    \item[(ii)] The $\yhat_{k+1}$ returned by Algorithm~\ref{alg:dual-gradient-descent} satisfies
    $$
    \dist\big(\yhat_{k+1}, \Ycalhat(x_k,\alpha_k)\big) \leq \nu_{k}\left(\frac{2\epsilon_k}{\alpha_k}\right)^{\rho_{k}/2}.
    $$
\end{itemize}
\end{lemma}

\begin{proof}
To prove part (i), let $\yhat^*_k \in \Ycal(x_k,\alpha_k)$.  We may now define $y^* = \proj{\Ycalhat(x^*,\alpha^*)}{\yhat^*_k}$.  It now follows from~\eqref{dist-star} with $(\yhat,y) = (\yhat^*_k,y^*)$, Lemma~\ref{lemma:link}(i), the triangle inequality, and Assumption~\ref{ass.first} that
\begin{align*}
&\dist\big(\yhat_k^*,\Ycalhat(x^*,\alpha^*)\big) \\
&\leq \nu^*\norm{A\yhat_k^*-Ay^*}^{\rho^*}
= \nu^*\norm{\tfrac{x^*_k-u_k}{\alpha_k} - \tfrac{(x^*-u^*)}{\alpha^*} }^{\rho^*} 
= \nu^*\norm{\tfrac{x_k^*-x_k}{\alpha_k} + \grad f(x_k)-\grad f(x^*)}^{\rho^*} \\
&\leq \nu^*\left(\tfrac{\norm{x_k^*-x_k}}{\alpha_k} + L_g\norm{x_k-x^*}\right)^{\rho^*} 
\leq \nu^*\left(\tfrac{\norm{x_k^*-x^*}}{\alpha_k} + \left(L_g+\tfrac{1}{\alpha_k}\right)\norm{x_k-x^*}\right)^{\rho^*},
\end{align*}    
which establishes the claim in part (i) to be true.

We now prove part (ii). If we define $h_k(z) := \frac{\alpha_k}{2}\norm{z}^2 + u_k^Tz$, then it follows that $\phid(\yhat;x_k,\alpha_k) = -h_k(A\yhat)$. Also, to simply notation, let us define $p_{k+1} := \proj{\Ycalhat(x_k,\alpha_k)}{\yhat_{k+1}}$. Using this notation, we can observe from the properties that the $\yhat_{k+1}$ returned by Algorithm~\ref{alg:dual-gradient-descent} must satisfy that
$$
\phid(p_{k+1};x_k,\alpha_k) - \phid(\yhat_{k+1};x_k,\alpha_k) 
\leq \phi(x_{k+1};x_k,\alpha_k) - \phid(\yhat_{k+1};x_k,\alpha_k)
\leq \epsk.
$$
Combining this inequality with strong convexity of $h_k$ and $p_{k+1} \in \Ycalhat(x_k,\alpha_k)$ gives
$$
\begin{aligned}
\epsk
 & \geq \phid(p_{k+1};x_k,\alpha_k) - \phid(\yhat_{k+1};x_k,\alpha_k) 
 = h_k(A\yhat_{k+1})- h_k(Ap_{k+1})\\
 &\geq \grad h_k(Ap_{k+1})^TA(\yhat_{k+1}-p_{k+1}) + \tfrac{\alpha_k}{2}\norm{A\yhat_{k+1}-Ap_{k+1}}^2\\
 &\geq -\grad \phid(p_{k+1};x_k,\alpha_k)^T(\yhat_{k+1}-p_{k+1}) + \tfrac{\alpha_k}{2}\norm{A\yhat_{k+1}-Ap_{k+1}}^2
 \geq \tfrac{\alpha_k}{2}\norm{A\yhat_{k+1}-Ap_{k+1}}^2.
\end{aligned}
$$
The previous inequality, Lemma~\ref{leamma:feasbound} with $(\yhat,y) = (\yhat_{k+1},p_{k+1})$, and Lemma~\ref{lem:alpha-fixed} gives
$$
\dist\big(\yhat_{k+1}, \Ycalhat(x_k,\alpha_k)\big)
\leq \nu_k \big( \|A\yhat_{k+1}-Ap_{k+1}\|^2\big)^{\rho_k/2}
\leq \nu_{k}\left(\tfrac{2\epsilon_k}{\alpha_k}\right)^{\rho_{k}/2},
$$
thus completing the proof of part (ii).
\end{proof}

For the remainder of this section, our analysis applies to two difference scenarios that are defined below based on the rate at which $\{\epsilon_k\} \to 0$. The first scenario uses 
\begin{equation}\label{def:theta}
\theta := (1-\alpha_{0}\mu_f)\in [\eta,1),  
\end{equation}
with the inclusion following from Lemma~\ref{lem-epsk-to-zero} and $\mu_f\leq L_g$ since they imply that $1 > \theta \equiv 1-\alpha_{0}\mu_f \geq 1 - ((1-\eta)/L_g) L_g = \eta$.  We can now state our two scenarios.

\begin{strat}\label{s1}
For some $\psi \in(0,1)$, we have $\epsilon_k \leq \min\{\tfrac{\alpha_0}{2},  \psi^{2k} \theta^{2(k+1)}\}$ for all $k\in\N{}$. 
\end{strat}

\begin{strat}\label{s2}
For some $\omega\in(0,1)$, we have $\epsilon_{k+1} \leq \omega^2 \epsilon_{k}$ for all $k\in\N{}$ with $\epsilon_0 \leq \alpha_0/2$. By applying this bound recursively, it follows that $\epsilon_k \leq \omega^{2k} \epsilon_0 \leq \alpha_0/2$ for all $k\in\N{}$.
\end{strat}

We consider both of these scenarios for the following reasons.  To implement Strategy~\ref{s1} the value of $\theta$ must be known, which means that the strong convexity parameter $\mu_f$ must be known. (Of course, in practice, one could also attempt to estimate $\mu_f$.) If this parameter is known, this strategy may be a good choice.  On the other hand, Strategy~\ref{s2} can be implemented without any knowledge of the strong convexity parameter, which is an advantage, but we shall see that our support identification result depends on the size of $\omega$ relative to $\theta$. Finally, let us note that implementing either strategy requires controlling the size of each element of the sequence $\{\epsilon_k\}$. This can be done for each $k\in\N{}$ by checking the value of $\epsilon_k$ returned by Algorithm~\ref{alg:dual-gradient-descent} when called by Algorithm~\ref{alg:main}.  If the value returned is not small enough, then Algorithm~\ref{alg:main} could recall Algorithm~\ref{alg:dual-gradient-descent} but this time using as an initial iterate the value returned the previous time it was called. This ``recall if necessary" procedure can be repeated as necessary until a sufficiently small $\epsilon_k$ is obtained as required by the chosen strategy.

We now prove a rate of convergence for both $x_k^*$ and $x_{k+1}$ to the solution $x^*$.

\begin{lemma}\label{lemma:linear-convergence}
Let Assumption~\ref{ass.first} and Assumption~\ref{ass:active-set} hold.  It follows that
\begin{equation}\label{eq:rate-of-xk-convergence}
\norm{x_{k+1}-x^*}\leq \theta^{k+1}\left(\norm{x_{0}-x^*} + \sqrt{2\alpha_0}\sum_{i=0}^{k}\frac{\sqrt{\epsilon_i}}{\theta^{i+1}}\right) \text{ for all } k\in\N{}
\end{equation}
where $\theta$ is defined in~\eqref{def:theta}.  Consequently, the following rates of convergence hold:
\begin{itemize}
\item[(i)] 
If the sequence $\{\epsilon_k\}$ satisfies  Strategy~\ref{s1}, then 
\begin{align}
   \norm{x_{k+1}-x^*}
   &\leq \theta^{k+1}\left(\norm{x_0-x^*} + \sqrt{2\alpha_0}/(1-\psi)\right) \  \text{for all}  \ k\in\N{} \ \text{and} \label{rates-11}\\
   \norm{x^*_k-x^*} 
   &\leq \theta^{k+1}\left(\norm{x_0-x^*} + \sqrt{2\alpha_0}/(1-\psi) +  \sqrt{2\alpha_0}\psi\right) \  \text{for all}  \ k\in\N{}. \label{rates-12}
\end{align}
\item[(ii)] 
If the sequence $\{\epsilon_k\}$ satisfies  Strategy~\ref{s2}, then the following hold: 
  \begin{itemize}
    \item[(a)] If $\omega < \theta$, then \begin{align}
      \norm{x_{k+1}-x^*}
      &\leq
      \theta^{k+1}\Big(\norm{x_0-x^*} + \tfrac{\sqrt{2\alpha_0\epsilon_0}}{(\theta-\omega)}\Big) \ \ \text{for all $k\in\N{}$,} \label{rates-2a1} \\
         \norm{x^*_k-x^*} 
      &\leq
      \theta^{k+1}\Big(\norm{x_0-x^*} + \tfrac{\sqrt{2\alpha_0\epsilon_0}}{(\theta-\omega)} + \tfrac{\sqrt{2\alpha_0\epsilon_0}}{\omega}\Big) \ \ \text{for all $k\in\N{}$.} \label{rates-2a2}
    \end{align}  
    \item[(b)] If $\omega > \theta$, then
    \begin{align}
       \norm{x_{k+1}-x^*}
       &\leq \omega^{k+1}\big(\norm{x_0-x^*} + \tfrac{\sqrt{2\alpha_0\epsilon_0}}{\omega -\theta}\big) \ \ \text{for all $k\in\N{}$ and} \label{rates-2b1} \\
       \norm{x^*_k-x^*}
       &\leq 
       \omega^{k+1}\big(\norm{x_0-x^*} + \tfrac{\sqrt{2\alpha_0\epsilon_0}}{\omega -\theta} + \tfrac{\sqrt{2\alpha_0\epsilon_0}}{\omega}\big) \ \ \text{for all $k\in\N{}$}. \label{rates-2b2}
    \end{align}   
    \item[(c)] If $\omega = \theta$, then
    \begin{align}
      \norm{x_{k+1}-x^*}
      &\leq
      \theta^{k+1}\big(\norm{x_0-x^*} + (k+1)\tfrac{\sqrt{2\alpha_0\epsilon_0}}{\theta}\big) \ \ \text{for all $k\in\N{}$ and} \label{rates-2c1} \\
      \norm{x^*_k-x^*}
      &\leq
      \theta^{k+1}\big(\norm{x_0-x^*} + (k+1)\tfrac{\sqrt{2\alpha_0\epsilon_0}}{\theta} + \sqrt{2\alpha_0\epsilon_0}\omega^k \big)\ \ \text{for all $k\in\N{}$.} \label{rates-2c2}
    \end{align}
  \end{itemize}  
\end{itemize}
\end{lemma}
\begin{proof}
We first prove~\eqref{eq:rate-of-xk-convergence} using a procedure similar to~\cite[Proposition 3]{schmidt2011}. It follows from Lemma~\ref{lem-epsk-to-zero}, optimality of $x^*$, and defining $p_k = \prox{\alpha^* r}{x_k-\alpha^*\grad f(x_k)}$ that
\begin{equation}\label{eq:xkbound1}
    \norm{x_{k+1}-x^*}^2  
    = \norm{x_{k+1}- p_k + p_k -\prox{\alpha^* r}{x^*-\alpha^*\grad f(x^*)}}^2.
\end{equation}
This equality, the Cauchy-Schwarz inequality, and non-expansivity of the proximal operator yields
\begin{align*}  
    \norm{x_{k+1}-x^*}^2 
    &\leq \norm{x_{k+1} - p_k}^2 + \norm{p_k -\prox{\alpha^* r}{x^*-\alpha^*\grad f(x^*)}}^2 \\
    &\quad + 2\norm{x_{k+1}- p_k}\norm{p_k - \prox{\alpha^* r}{x^*-\alpha^*\grad f(x^*)}} \\
    &\leq \norm{x_{k+1} - p_k}^2 + \norm{x_k - x^* - \alpha^*(\grad f(x_k) - \grad f(x^*))}^2 \\
    &\quad + 2\norm{x_{k+1} - p_k}\norm{x_k - x^* - \alpha^*(\grad f(x_k) - \grad f(x^*))}. 
\end{align*}
Combining this inequality, Lemma~\ref{lemma:chibound}, and the fact that, for all $k\in\N{}$, we know $\alpha_k = \alpha^*$ and $x_{k+1}=\xhat_{k+1}$ returned by Algorithm~\ref{alg:dual-gradient-descent} satisfy $x_{k+1} \in\Tepsk(x_k,\alpha^*)$ yields
\begin{align*}
 \norm{x_{k+1}-x^*}^2 
 &\leq 2\alpha^*\epsilon_k + \norm{x_k-x^*-\alpha^*(\grad f(x_k)-\grad f(x^*))}^2 \\
 &\quad + 2\sqrt{2\alpha^*\epsilon_k} \norm{x_k-x^*-\alpha^*(\grad f(x_k)-\grad f(x^*))}.
\end{align*}
To bound the norm that appears in the previous inequality, we may use the Cauchy-Schwarz inequality and~\cite[Theorem 2.1.12]{nesterov2004introductory} to obtain
\begin{align}
    &\|x_k - x^* - \alpha^*(\grad f(x_k)-\grad f(x^*))\|^2 \nonumber\\ 
    &= \norm{x_k-x^*}^2 + (\alpha^*)^2\norm{\grad f(x_k)-\grad f(x^*)}^2 - 2\alpha^*(x_k-x^*)^T\left(\grad f(x_k)-\grad f(x^*)\right) \nonumber\\
    &\leq \norm{x_k-x^*}^2 + (\alpha^*)^2\norm{\grad f(x_k)-\grad f(x^*)}^2  \nonumber \\
    &\quad - 2\alpha^*\left(\tfrac{L_g\mu_f}{L_g+\mu_f}\norm{x_k-x^*}^2 + \tfrac{1}{\mu_f+L_g}\norm{\grad f(x_k)-\grad f(x^*)}^2\right)\nonumber\\
    &= \left(1 - \tfrac{2\alpha^* L_g\mu_f}{L_g+\mu_f}\right)\norm{x_k-x^*}^2 + \alpha^* \left(\alpha^*- \tfrac{2}{L_g+\mu_f}\right)\norm{\grad f(x_k)-\grad f(x^*)}^2.\label{eq:xkbound2}
\end{align}
From $\eta \in (0,1)$ and Lemma~\ref{lem-epsk-to-zero} we have $\alpha^*\in(0,(1-\eta)/L_g] \subset (0,1/L_g)$.  Combining this with $\mu_f \leq L_g$ we have $\alpha^* < 2/(L_g+\mu_f)$, and from strong convexity of $f$ that $\norm{\grad f(x_k)-\grad f(x^*)}^2\geq \mu_f^2\norm{x_k-x^*}^2$. Combining these two facts with \eqref{eq:xkbound2} gives
\begin{equation}\label{eq:xkbound3}
\begin{aligned}
     \|x_k-x^*&-\alpha^*(\grad f(x_k)-\grad f(x^*))\|^2 \\
     &\leq \left(1 - \tfrac{2\alpha^* L_g\mu_f}{L_g+\mu_f}\right)\norm{x_k-x^*}^2 + \alpha^* \mu_f^2\left(\alpha^*- \tfrac{2}{L_g+\mu_f}\right)\norm{x_k-x^*}^2 \\
     &= \Big(1 - \Big(\tfrac{2\alpha^* L_g\mu_f}{L_g+\mu_f} + \tfrac{2\alpha^* \mu_f^2}{L_g+\mu_f}\Big) + (\alpha^*)^2\mu_f^2\Big)\norm{x_k-x^*}^2 \\
     &= (1-\alpha^*\mu_f)^2\norm{x_k-x^*}^2 
     = \theta^2\norm{x_k-x^*}^2,
\end{aligned}
\end{equation}
where the last equality follows from Lemma~\ref{lem-epsk-to-zero}. Combining \eqref{eq:xkbound1} and \eqref{eq:xkbound3} gives
$$
\begin{aligned}
    \norm{x_{k+1}-x^*}^2
    \leq \theta^2\norm{x_k-x^*}^2 + 2\sqrt{2\alpha^*\epsilon_k}\theta\norm{x_k-x^*} + 2\alpha^*\epsilon_k 
    = \big(\theta\norm{x_k-x^*} + \sqrt{2\alpha^*\epsilon_k}\, \big)^2.
\end{aligned}
$$
Taking the square root of the previous inequality and applying it  recursively shows 
$$
\norm{x_{k+1}-x^*}\leq \theta^{k+1}\norm{x_0-x^*} + \sum_{i=0}^{k}\theta^{k-i}\sqrt{2\alpha_i\epsilon_i} \ \ \text{for all} \ \ k\in\N{}.
$$
If we now use basic algebra and the fact that $\alpha_i = \alpha_0$ for all $i\in\N{}$, we arrive at 
$$
\norm{x_{k+1}-x^*}\leq \theta^{k+1}\left(\norm{x_0-x^*} + \sqrt{2\alpha_0} \sum_{i=0}^{k}\frac{\sqrt{\epsilon_i}}{\theta^{i+1}}\right) \ \ \text{for all} \ \ k\in\N{},
$$
which proves that the inequality in \eqref{eq:rate-of-xk-convergence} holds.

Before proving the remaining results, let us make a few observations.  First, since  $\alpha_k = \alpha_0$ for all $k\in\N{}$ (see Lemma~\ref{lem-epsk-to-zero}) it follows from the construction of Algorithm~\ref{alg:main} that $x_{k+1} = \xhat_{k+1}$ for all $k\in\N{}$.  Combining this fact with the triangle inequality, Lemma~\ref{lemma:chibound} with  $x_{k+1} = \xhat_{k+1} \in \Tepsk(x_k,\alpha_k)$ (which holds by construction of Algorithm~\ref{alg:main}), and the fact that $\alpha_k = \alpha_0$ for all $k\in\N{}$ shows that
\begin{equation}\label{eq:xkstar-xk}
    \norm{x_k^*-x^*}
    \leq \norm{x_k^*-x_{k+1}} + \norm{x_{k+1}-x^*}
    \leq \sqrt{2\alpha_0\epsk} + \norm{x_{k+1}-x^*}.
\end{equation}
We now consider specific choices for the sequence $\{\epsilon_k\}$ to derive the remaining results.

\smallskip
\noindent\textbf{Part (i).} Let $\{\epsilon_k\}$ satisfy the bound in Strategy~\ref{s1}. Applying this bound on $\epsilon_k$ for all $k\in\N{}$ to the right-hand side of \eqref{eq:rate-of-xk-convergence} leads to a  geometric sum with factor $\psi$, from which~\eqref{rates-11} follows.  Combining~\eqref{rates-11} with the assumed bound on $\epsilon_k$ for all $k\in\N{}$, \eqref{eq:xkstar-xk}, and $\psi\in(0,1)$ proves that~\eqref{rates-12} holds.

\smallskip
\noindent\textbf{Part (ii)}.
Let $\{\epsilon_k\}$ satisfy Strategy~\ref{s2} so that $\epsk \leq \omega^{2k}\epsilon_0$. This fact and~\eqref{eq:rate-of-xk-convergence} gives
\begin{equation}\label{eq:rate-cases}
    \norm{x_{k+1}-x^*}
    \leq \theta^{k+1}\left(\norm{x_0-x^*} + \frac{\sqrt{2\alpha_0\epsilon_0}}{\theta}\sum_{i=0}^{k}\left(\frac{\omega}{\theta}\right)^i\right).
\end{equation}
We can now consider the three sub-parts (a), (b), and (c) in turn.

To prove part (ii)(a), suppose that $\omega < \theta$. It follows that 
$\sum_{i=0}^{k}\left(\omega/\theta\right)^i \leq \sum_{i=0}^{\infty}\left(\omega/\theta\right)^i
=\frac{1}{1-\omega/\theta}
$. 
Combining this inequality with~\eqref{eq:rate-cases} shows that~\eqref{rates-2a1} holds.  Combining~\eqref{rates-2a1} with~\eqref{eq:xkstar-xk}, $\epsk \leq \omega^{2k}\epsilon_0$, and $\omega < \theta$ gives~\eqref{rates-2a2}, completing this proof.

To prove part (ii)(b), suppose that $\omega > \theta$.  Let us first observe by using basic algebra and the fact that a geometric series with factor $\theta/\omega \in (0,1)$ is finite that
$$
\theta^{k+1}\sum_{i=0}^{k}\left(\frac{\omega}{\theta}\right)^i
= \omega^{k+1}\sum_{i=1}^{k+1}\left(\frac{\theta}{\omega}\right)^{i}
= \omega^k\theta \sum_{i=0}^k \left(\frac{\theta}{\omega}\right)^i
\leq \omega^k\theta \sum_{i=0}^\infty \left(\frac{\theta}{\omega}\right)^i
= \frac{\omega^k\theta}{1-\theta/\omega}
= \frac{\omega^{k+1}\theta}{\omega-\theta}.
$$
This inequality can be combined with~\eqref{eq:rate-cases} and $\omega > \theta$ to conclude that 
$$
\norm{x_{k+1}-x^*}
\leq \theta^{k+1}\norm{x_0-x^*} + \omega^{k+1}\frac{\sqrt{2\alpha_0\epsilon_0}}{
\omega-\theta}
\leq \omega^{k+1}\left(\norm{x_0-x^*} + \frac{\sqrt{2\alpha_0\epsilon_0}}{\omega -\theta}\right),
$$
thus proving that~\eqref{rates-2b1} holds.  Then, combining~\eqref{rates-2b1} with~\eqref{eq:xkstar-xk} and $\epsk \leq \omega^{2k}\epsilon_0$ gives~\eqref{rates-2b2}.

To prove part (ii)(c), suppose that $\omega = \theta$.  In this case, we know that $\sum_{i=0}^{k}\left(\frac{\omega}{\theta}\right)^i = k+1$.  Combining this equality with~\eqref{eq:rate-cases} establishes the inequality in~\eqref{rates-2c1}.  Then, combining~\eqref{rates-2c1} with~\eqref{eq:xkstar-xk} and $\epsk \leq \omega^{2k}\epsilon_0$ gives~\eqref{rates-2c2}, thus completing the proof.
\end{proof}

\begin{lemma}\label{lemma:utltimate}
Let Assumption~\ref{ass:active-set} and Assumption~\ref{ass:nd} hold.  Then, for all $k\in\N{}$, the vector $\yhat_{k+1}$  returned by Algorithm~\ref{alg:dual-gradient-descent} satisfies the following bounds:
\begin{itemize}
    \item[(i)]
    If the sequence $\{\epsilon_k\}$ satisfies  Strategy~\ref{s1}, then
    \begin{align*}
    \dist\big(\yhat_{k+1},\Ycalhat(x^*,\alpha^*)\big) 
    &\equiv
    \| \yhat_{k+1}-\proj{\Ycalhat(x^*,\alpha^*)}{\yhat_{k+1}} \| \\
    &= \Ocal\big(\big[\theta^{\min\{\rho_{\min}, \rho^*\}}\big]^k\big) \ \ \text{for all $k\in\N{}$.}
    \end{align*}
    \item[(ii)] 
    If the sequence $\{\epsilon_k\}$ satisfies  Strategy~\ref{s2}, then
    \begin{align*}
    \dist\big(\yhat_{k+1},\Ycalhat(x^*,\alpha^*)\big) 
    &\equiv
    \norm{\yhat_{k+1}-\proj{\Ycalhat(x^*,\alpha^*)}{\yhat_{k+1}}} \\
    &= 
    \begin{cases}
\Ocal\left(\big[\max\{\omega^{\rho_{\min}},\theta^{\rho^*}\}\big]^k\right) & \text{ if } \omega<\theta,  \\
\Ocal\left([\omega^{\min\{\rho_{\min}, \rho^*\}}]^k\right) & \text{ if } \omega>\theta, \\
\Ocal\left((k\theta^k)^{\min\{\rho_{\min}, \rho^*\}}\right) & \text{ if } \omega=\theta.
\end{cases}
    \end{align*}
    \end{itemize}
\end{lemma}
\begin{proof}
Let us define $p_{k+1} = \proj{\Ycalhat(x_k,\alpha_k)}{\yhat_{k+1}}$ and $p_{k+1}^* = \proj{\Ycalhat(x^*,\alpha^*)}{p_{k+1}}$. The triangle inequality and non-expansiveness of the projection operator shows that
\begin{align}
 &\norm{\yhat_{k+1} - \proj{\Ycalhat(x^*,\alpha^*)}{\yhat_{k+1}}} \nonumber \\
 &= \|\yhat_{k+1}- p_{k+1} + p_{k+1} 
 - p_{k+1}^* + p_{k+1}^* - \proj{\Ycalhat(x^*,\alpha^*)}{\yhat_{k+1}}\|\nonumber\\
 &\leq \| \yhat_{k+1}- p_{k+1} \| + \|p_{k+1} - p_{k+1}^*\| + \|p_{k+1}^* - \proj{\Ycalhat(x^*,\alpha^*)}{\yhat_{k+1}}\| \nonumber\\
 &\leq 
 2\| \yhat_{k+1}- p_{k+1} \| + \|p_{k+1} - p_{k+1}^* \| 
 = 2\, \dist\big(\yhat_{k+1},\Ycalhat(x_k,\alpha_k)\big) + \dist\big(p_{k+1},\Ycalhat(x^*,\alpha^*)\big). \nonumber
\end{align}
This inequality, Lemma~\ref{lemma:two-bounds}, and  $\alpha_k = \alpha^*$ for all $k\in\N{}$ (see Lemma~\ref{lem-epsk-to-zero}) 
show that
\begin{equation}\label{eq:approx-dual-sol-distance-pre}
\norm{\yhat_{k+1} - \proj{\Ycalhat(x^*,\alpha^*)}{\yhat_{k+1}}}
\leq 
2\nu_{k}\left(\tfrac{2\epsilon_k}{\alpha^*}\right)^{\tfrac{\rho_{k}}{2}} + \nu^*\Big(\tfrac{\norm{x_k^*-x^*}}{\alpha^*} + (L_g+\tfrac{1}{\alpha^*})\norm{x_k-x^*}\Big)^{\rho^*}.
\end{equation}
Note that regardless of which case in the statement of the lemma we are considering (i.e., Strategy~\ref{s1} for part (i) or Strategy~\ref{s2} for part (ii)), we know that $\epsilon_k \leq \alpha_0/2 \equiv \alpha^*/2$ for all $k\in\N{}$.  This bound, \eqref{eq:approx-dual-sol-distance-pre}, and Assumption~\ref{ass:upbd} together show that
\begin{equation}\label{eq:approx-dual-sol-distance}
\norm{\yhat_{k+1} - \proj{\Ycalhat(x^*,\alpha^*)}{\yhat_{k+1}}}
\leq 
2\nu_{\max}\left(\tfrac{2\epsilon_k}{\alpha^*}\right)^{\tfrac{\rho_{\min}}{2}} + \nu^*\Big(\tfrac{\norm{x_k^*-x^*}}{\alpha^*} + (L_g+\tfrac{1}{\alpha^*})\norm{x_k-x^*}\Big)^{\rho^*}.
\end{equation}
We can now consider the two cases in the statement of the lemma.

For part (i), $\{\epsilon_k\}$ satisfies Strategy~\ref{s1} so that $\epsilon_k \leq \min\{\alpha_0/2,  \psi^{2k} \theta^{2(k+1)}\}$ for all $k\in\N{}$.  Combining this bound, \eqref{eq:approx-dual-sol-distance}, Lemma~\ref{lemma:linear-convergence}(i), $\theta\in[\eta,1)$, and $\psi\in(0,1)$ gives
\begin{align*}
\norm{\yhat_{k+1}-\proj{\Ycalhat(x^*,\alpha^*)}{\yhat_{k+1}}}
&=
\Ocal\big[\big(\psi^k\theta^{k+1}\big)^{\rho_{\min}}\big] + \big[\Ocal(\theta^{k+1}) + \Ocal(\theta^{k})\big]^{\rho^*} \\
&= \Ocal\big([\theta^{\rho^{\min}}]^k\big) + \Ocal\big([\theta^{\rho^*}]^k\big)
= \Ocal\big(\big[\theta^{\min\{\rho_{\min}, \rho^*\}}\big]^k\big),
\end{align*}
which completes the proof for part (i) of this lemma.

For part (ii),  $\{\epsilon_k\}$ satisfies Strategy~\ref{s2} so that $\epsilon_k \leq \omega^{2k}\epsilon_0$ for all $k\in\N{}$.  This bound,   \eqref{eq:approx-dual-sol-distance}, and Lemma~\ref{lemma:linear-convergence} yield
$$
\begin{aligned}
&\norm{\yhat_{k+1}-\proj{\Ycalhat(x^*,\alpha^*)}{\yhat_{k+1}}} \\
&=
\begin{cases}
\Ocal\big[\left(\omega^{k}\right)^{\rho_{\min}}\big] + \left[\Ocal\left(\theta^{k+1}\right) + \Ocal(\theta^{k})\right]^{\rho^*} & \text{ if } \omega<\theta, \\ 
\Ocal\big[\left(\omega^k\right)^{\rho_{\min}}\big] + \left[\Ocal\left(\omega^{k+1}\right) + \Ocal(\omega^{k})\right]^{\rho^*} & \text{ if } \omega>\theta, \\
\Ocal\big[\left(\omega^k\right)^{\rho_{\min}}\big] + \left[\Ocal\left((k+1)\theta^{k+1}\right) + \Ocal(k\theta^{k})\right]^{\rho^*} & \text{ if } \omega=\theta,
\end{cases}
\\
&=
\begin{cases}
\Ocal\left(\big[\max\{\omega^{\rho_{\min}},\theta^{\rho^*}\}\big]^k\right) & \text{ if } \omega<\theta,  \\
\Ocal\left([\omega^{\min\{\rho_{\min}, \rho^*\}}]^k\right) & \text{ if } \omega>\theta, \\
\Ocal\left((k\theta^k)^{\min\{\rho_{\min}, \rho^*\}}\right) & \text{ if } \omega=\theta,
\end{cases}
\end{aligned}
$$
which completes the proof for part (ii) of this lemma.
\end{proof}

We are now ready to present our main support identification theorem.

\begin{theorem}
Let Assumption~\ref{ass:active-set}--Assumption~\ref{ass:upbd} hold, and define
$$
\Theta
:= 
\begin{cases}
\min\{1, \min_{i\in\Scal(x^*)} \norm{[x^*]_{g_i}} \} & \text{if $\Scal(x^*)\neq\emptyset$,} \\
1 & \text{otherwise.}
\end{cases}
$$
Then, the following results hold: 
\begin{itemize}
    \item[(i)] If $\{\epsilon_k\}$ satisfies Strategy~\ref{s1}, 
    then  $\Scal(x_{k+1})=\Scal(x^*)$ for all $k\geq K_1$ with 
    $$
    K_1 := \max\left\{\Ocal\left(\frac{\log \Theta}{\log \theta}\right), \Ocal\left(\frac{\log \delta^*}{\log \max\left\{\left[\theta^{\min\{\rho_{\min},\rho^*\}}\right],[\Psi\theta]^{2\iota}\right\}}\right)\right\}. 
    $$
    \item[(ii)] If $\{\epsilon_k\}$ satisfies Strategy~\ref{s2}, 
    then 
    $\Scal(x_{k+1})=\Scal(x^*)$ for all $k\geq K_2$ with
    $$
    K_2 :=  \begin{cases}
            \max\left(
            \Ocal\left(\frac{\log \Theta}{\log\theta}\right), 
            \Ocal\left(\frac{\log {\delta^*}}{\log\left(\max\{\omega^{\rho_{\min}},\theta^{\rho^*},\omega^{2\iota}\}\right)}\right)
            \right) & \text{ if } \omega<\theta,\\
             \max\left(
             \Ocal\left(\frac{\log \Theta}{\log\omega}\right), 
             \Ocal\left(\frac{\log {\delta^*}}{\log\left(\max\{\omega^{\min\{\rho_{\min}, \rho^*\}}, \omega^{2\iota}\}\right)}\right)
             \right) & \text{ if } \omega>\theta, \\
           \max\left(\Ocal(C_{\Theta}), \Ocal(C_{\delta^*})\right) & \text{ if } \omega=\theta,
        \end{cases}
$$
$C_{\Theta}$ is the smallest nonnegative integer such that $(k+1)\omega^{k+1}< \Theta$ for all $k\geq C_{\Theta}$, and $C_{\delta^*}$ is the smallest nonnegative integer such that $(k\omega^k)^{\min\{\rho_{\min},\rho^*\}} + \omega^{2\iota k}< \delta^*$ for all $k\geq C_{\delta^*}$.
\end{itemize} 
\end{theorem}

\begin{proof}
The proof begins by establishing two claims, which we consider one at a time.

\medskip
\noindent
\textbf{Claim 1.} We claim that   $\norm{[\yhat_{k+1}]_{\Mcal(i)}}<[\lambda]_i-\epsilon_{k-1}^{\iota}$ and $[x_{k+1}]_{g_i} = 0$ for all $i\not\in\Scal(x^*)$ and $k\geq k_1$ with
    \begin{equation}\label{def:k1}
    k_1 = 
    \begin{cases}
       \Ocal\left(\frac{\log \delta^*}{\log \max\left\{\left[\theta^{\min\{\rho_{\min},\rho^*\}}\right],[\Psi\theta]^{2\iota}\right\}}\right) & \text{if Strategy~\ref{s1} is used,}\\ 
      \Ocal\left(\frac{\log {\delta^*}}{\log\left(\max\{\omega^{\rho_{\min}},\theta^{\rho^*}\omega^{2\iota}\}\right)}\right)&\text{if Strategy~\ref{s2} is used and $\omega<\theta$,}\\
      \Ocal\left(\frac{\log {\delta^*}}{\log\left(\max\{\omega^{\min\{\rho_{\min}, \rho^*\}} \omega^{2\iota}\}\right)}\right) &\text{if Strategy~\ref{s2} is used and $\omega>\theta$,}\\
     \Ocal(C_{\delta^*})&\text{if Strategy~\ref{s2} is used and $\omega=\theta$.}
    \end{cases}
    \end{equation}
    To prove this, note that if $\Scal(x^*) = [\ngrp]$ (i.e., if no $i\notin\Scal(x^*)$ exists), then the claim trivially holds for all $k \geq 0$, which agrees with the definition of $k_1$ since $\delta^* = 1$ and $C_{\delta^*} = 0$ in this case.   Therefore, for the remainder of the proof, we assume that $\Scal(x^*) \subsetneqq [\ngrp]$.  It  follows from Lemma~\ref{lemma:utltimate} that
    \begin{align*}
    &\dist(\yhat_{k+1}, \Ycalhat(x^*,\alpha^*)) + \epsilon_{k-1}^{\iota} \\
    &= 
    \begin{cases}
    \Ocal\big(\big[\theta^{\min\{\rho_{\min}, \rho^*\}}\big]^k\big) + 
    \Ocal\left(\left([\Psi\theta]^{2\iota}\right)^{k}\right) & \text{if Strategy~\ref{s1} is used,}\\ 
    \Ocal\left(\big[\max\{\omega^{\rho_{\min}},\theta^{\rho^*}\}\big]^k\right) +\Ocal\left(\left(\omega^{2\iota}\right)^{k}\right)&\text{if Strategy~\ref{s2} is used and $\omega<\theta$,}\\
    \Ocal\left([\omega^{\min\{\rho_{\min}, \rho^*\}}]^k\right) +\Ocal\left(\left(\omega^{2\iota}\right)^{k}\right)&\text{if Strategy~\ref{s2} is used and $\omega>\theta$,}\\
    \Ocal\left((k\theta^k)^{\min\{\rho_{\min}, \rho^*\}}\right) +\Ocal\left(\left(\omega^{2\iota}\right)^{k}\right)&\text{if Strategy~\ref{s2} is used and $\omega=\theta$.}
    \end{cases}
    \end{align*}
    Thus, there exists a constant $k_1$ satisfying~\eqref{def:k1} that makes the right-hand side of the previous inequality less than $\delta^*$, and therefore both $\epsilon_{k-1}^{\iota} < \delta^*$ and $\dist(\yhat_{k+1}, \Ycalhat(x^*,\alpha^*))  < \delta^* -  \epsilon_{k-1}^{\iota}$ must hold for all $k \geq k_1$.  This second inequality implies that, for each $k\geq k_1$, there exists $p^*_k \in \Ycalhat(x^*,\alpha^*)$ such that $\|\yhat_{k+1}- p_k^*\|  < \delta^* -  \epsilon_{k-1}^{\iota}$. 
    This inequality and Assumption~\ref{ass:nd} together 
    imply
    $$
    \| [\yhat_{k+1}]_{\Mcal(i)}\| 
    \leq \|[\yhat_{k+1}]_{\Mcal(i)} - [p_k^*]_{\Mcal(i)}\| + \|[p_k^*]_{\Mcal(i)}\| 
    < \delta^*-\epsilon_k^{\iota} + [\lambda]_i - \delta^*
    = [\lambda]_i - \epsilon_{k-1}^{\iota}
    $$
    for all $i\not\in\Scal(x^*)$ and $k\geq k_1$, which proves the first part of the claim.

To prove the second part of the claim, first note from how $\yhat_{k+1}$ is computed in Algorithm~\ref{alg:dual-gradient-descent} that there exists an integer $t_k>0$ such that $\yhat_{k+1} = \yhat_{k,t_k}$. Combining this fact with the first part of the claim that we just proved, we find that $\norm{[\yhat_{k+1}]_{\Mcal(i)}} =\norm{[\yhat_{k, t_k}]_{\Mcal(i)}} <[\lambda]_i-\epsilon_{k-1}^{\iota}$ for all $i\notin\Scal(x^*)$ and $k\geq k_1$.  Using this inequality and Line~\ref{line:proj-primal} of Algorithm~\ref{alg:dual-gradient-descent}, it follows that $[\xhat_{k, t_k}]_{g_i} = 0$ for all $i\notin\Scal(x^*)$ and $k \geq k_1$. Thus,  for all $i\notin\Scal(x^*)$ and $k \geq k_1$, we have $[\xhat_{k+1}]_{g_i} = [\xhat_{k, t_k}]_{g_i} = 0$. This result and Assumption~\ref{ass:alpha} imply that $[x_{k+1}]_{g_i} = 0$ for all $i\not\in\Scal(x^*)$ and $k\geq k_1$ since $x_{k+1}=\xhat_{k+1}$.

\medskip
\noindent
\textbf{Claim 2.} 
 We claim that $[x_{k+1}]_{g_i} \neq 0$ for all  $i\in\Scal(x^*)$ and $k\geq k_2$ with
    \begin{equation}\label{def:k2}
          k_2=
        \begin{cases}
            \Ocal\left(\frac{\log \Theta}{\log \theta}\right) &\text{if (Strategy~\ref{s1} is used) or (Strategy~\ref{s2} is used and $\omega<\theta$),}\\  
            \Ocal\left(\frac{\log \Theta}{\log\omega}\right) &\text{if Strategy~\ref{s2} is used and $\omega>\theta$,}\\
            \Ocal(C_{\Theta}) &\text{if Strategy~\ref{s2} is used and $\omega=\theta$.}
        \end{cases}  
    \end{equation}
To prove this, note that if $\Scal(x^*) = \emptyset$, then the claim trivially holds for all $k \geq 0$, which agrees with the definition of $k_2$ since $\Theta = 1$ and $C_{\Theta} = 0$ in this case.   Thus, for the remainder of the proof, we assume that $\Scal(x^*) \neq \emptyset$.  It follows from Lemma~\ref{lemma:linear-convergence} that
    \begin{align*}
        &\norm{x_{k+1}-x^*} \\
        &=
        \begin{cases}
        \Ocal\left(\theta^{k+1}\right) & \text{if (Strategy~\ref{s1} is used) or (Strategy~\ref{s2} is used and $\omega<\theta$)},\\ 
        \Ocal\left(\omega^{k+1}\right)    &\text{if Strategy~\ref{s2} is used and $\omega>\theta$ hold,}\\
        \Ocal\left((k+1)\theta^{k+1}\right)    &\text{if Strategy~\ref{s2} is used and $\omega=\theta$ hold.}
        \end{cases}
    \end{align*}
    Thus, there exists a constant $k_2$ satisfying \eqref{def:k2} that makes the right-hand side of the above inequality less that $\Theta$.
    The fact that $\norm{x_{k+1}-x^*}< \Theta$ for $k\geq k_2$ implies $\norm{[x_{k+1}]_{g_i} - [x^*]_{g_i}} < \Theta$ for all $i\in\Scal(x^*)$ and $k\geq k_2$. This further suggests $[x_{k+1}]_{g_i}\neq 0$ for all $i\in\Scal(x^*)$ and $k \geq k_2$.

Parts (i) and (ii) follow since $\Scal(x_{k+1})=\Scal(x^*)$ for all $k \geq \max\{k_1,k_2\}$.
\end{proof}

We remark that if we do not make any assumption on  the rate of convergence of $\{\epsk\}$ to zero (e.g., do not assume that a strategy such as  Strategy~\ref{s1} or Strategy~\ref{s2} is used), then a small modification of the above analysis establishes that support identification still occurs.  However, we are no longer able to give a bound on the maximum number of iterations before support identification will occur.

\section{Numerical Results}\label{sec.numerical}

In this section, we discuss our numerical tests.  In Section~\ref{sec:problem} we discuss the formulation of the test problem, in Section~\ref{sec:implementation} we discuss the details of our implementations, and in Section~\ref{sec.results} we discuss the tests performed and the results of those tests.

\subsection{Problem Formulation} \label{sec:problem}

We focus on the learning task of binary classification.  To this end, we consider problem~\eqref{prob.main} where the function $f$ is the binary logistic loss defined as
$$
f(x) = \tfrac{1}{N} \sum_{i=1}^{N} \log \left(1+e^{-y_{i} x^{T} d_{i}}\right),
$$
where $d_i\in\R{n}$ is the $i$th data point, $N$ is the number of data points in the data set, and $y_i\in \{-1, 1\}$ is the class label for the $i$th data point. Such an approach is commonly used in machine learning applications.

Data sets for the logistic regression problem were obtained from the LIBSVM repository.\footnote{https://www.csie.ntu.edu.tw/cjlin/libsvmtools/datasets} We excluded all multi-class (greater than two) classification instances, instances with less than $100$ features, and all data sets that were too large ($\geq$ 8GB) to be loaded into memory. For the adult data (a1a--a9a) and webpage data (w1a--w8a), we used only the largest instances, namely a9a and w8a. This left us with the final set of $11$ data sets found in Table~\ref{tab:test-db}. If the source of the data indicated that a data set was scaled, we used it without modification.  When the website did not indicate that scaling for a data set was used, we scaled each column of the feature data (i.e., feature-wise scaling) into the range $[-1,1]$ by dividing its entries by the largest entry in absolute value. 

We consider the group regularizer $r$ defined in~\eqref{eq:natOverlap}. As for defining the groups, we use a similar strategy as that used in~\cite{villa2014proximal} to define different overlapping groups, where the overlapping of groups is controlled by parameters $\ratio{} \in(0,1)$ and $\grpsize{} \in\N{}$. The product of $\ratio{}$ and $\grpsize{}$ determines the number of elements in each group that overlap with its neighboring groups. For example, if  $n = 13$, $\ratio{} = 0.2$, $\grpsize{} =5$, then the groups are $g_1=\{1,2,3,4,5\}$, $g_2=\{5,6,7,8,9\}$, and $g_3=\{9,10,11,12,13\}$, which have overlap of size $1 = \ratio{} \cdot \grpsize{}$. In addition, we consider different solution sparsity levels, which are achieved by adjusting the parameters $\{[\lambda]_{i}\}_{i=1}^{\ngrp}$. Specifically, for each $i\in[\ngrp]$, we define $[\lambda]_i= \Lambda\sqrt{|g_i|}$ for some $\Lambda < \Lambda_{\min}$, where $\Lambda_{\min}$ is the minimum positive number such that the solution to the problem with $[\lambda]_i = \Lambda_{\min}\sqrt{|g_i|}$ for all $i\in[\ngrp]$ is $x=0$. Since $\Lambda_{\min}$ cannot be computed analytically, we test a range of values and choose the smallest one that gives the zero solution; as a starting guess for $\Lambda_{\min}$, we use the one derived from the non-overlapping \grplone{} case as discussed in~\cite[equation (23)]{Yang2015}. For the tests described later in this section, we make specific  choices for the parameters $\ratio{}$, $\grpsize{}$, and $\Lambda$. 

\begin{table}[thbp]
\footnotesize
\centering
\caption{The first column (data set) gives the name of the data set.  The second column (N) and third column (n) indicate the number of data points and number of features (equivalently, the number of optimization variables), respectively.  The fourth column (scale) provides the feature-wise scaling used: each feature is either scaled into the given interval or scaled to have mean zero ($\mu = 0$) and variance one ($\sigma^2=1$).  The fifth column (who) indicates whether the data set came pre-scaled from the LIBSVM website (``website"), or whether it did not come pre-scaled and we scaled it (``us") as described in Section~\ref{sec:problem}.}
\smallskip
{\begin{tabular}{|c|cc|cc|}
\hline
data set           & N      & n    & scale               & who      \\ \hline
a9a                & 32561  & 123  & {[}0,1{]}           & website    \\
colon-cancer       & 62     & 2000 & $(\mu,\sigma^2)= (0,1)$ & website     \\
duke breast-cancer & 44     & 7192 & $(\mu,\sigma^2)= (0,1)$ & website     \\
gisette            & 6000   & 5000 & {[}-1,1{]}          & website     \\
leukemia           & 38     & 7129 & $(\mu,\sigma^2)= (0,1)$ & website     \\
madelon            & 2000   & 500  & {[}-1,1{]}          & us          \\
mushrooms          & 8124   & 112  & {[}0,1{]}           & website     \\
news20.binary      & 19996  & 1355191 & {[}0,1{]}           & website     \\
rcv1.binary        & 20242  & 47236   & {[}0,1{]}           & website     \\
real-sim           & 72309  & 20958   & {[}0,1{]}           & website     \\
w8a                & 49749  & 300  & {[}0,1{]}           & website    \\
\hline
\end{tabular}}
\label{tab:test-db}
\end{table}

\subsection{Implementation details}\label{sec:implementation}  
 
In Algorithm~\ref{alg:main}, the values of the input parameters that we use are  $\xi=0.5, \eta=10^{-3}$ and $\zeta=0.8$, and the remaining parameters $\gamma_1$ and $\gamma_2$ will be described later since they are tuned. We set $x_0 \gets 0$ and $\alpha_0 \gets 1$. For both $\opt = \optone{}$ and $\opt = \opttwo{}$, Algorithm~\ref{alg:main} decreases the value of the PG parameter for the next iteration using a multiplicative factor when $\flagpgk = \decalpha$.  In practice, this strategy might be conservative as the PG parameter can become relatively small after encountering several consecutive iterations that trigger a decrease in the PG parameter. Instead, in our implementation, we update the PG parameter at the end of each iteration of Algorithm~\ref{alg:main} following an idea from  \cite{becker2011templates}. Specifically, for the $k$th iteration, if the inequality in Line~\ref{line:armijo-pg} is not satisfied with $j=0$ (meaning that no backtracking is performed), we increase the PG parameter as $\alpha_{k+1}\gets 1.1 \alpha_k$; otherwise, we decrease the PG parameter as $\alpha_{k+1} \gets \zeta \alpha_k$. While  this PG parameter update strategy works better than the basic strategy in Algorithm~\ref{alg:main}, it is not covered by our analysis in Section~\ref{sec.analysis}. However, a simple modification is to allow the PG parameter to increase a finite number of times, which would be covered by our analysis with a larger constant in the complexity result. We compare the two adaptive termination conditions $\optone{}$ and $\opttwo{}$ with the absolute termination condition from \cite{schmidt2011}, i.e., setting $\epsilon_{k}=\frac{\const{}}{k^{3}}$ for some positive number $\const$ (its value will be given later since it is tuned), which we will refer to as $\optthree{}$. For $\optthree{}$, since there is no guarantee that $\xhat_{k+1} - x_k$ is a descent direction for $f+r$, we use a strategy considered in \cite{schmidt2011} where the PG parameter is initially set to $\alpha_0$ and every time the inequality $f(\xhat_{k+1})\leq f(x_k)+ \nabla f(x_k)^T (\xhat_{k+1}-x_k) + \frac{1}{\alpha_k}\norm{\xhat_{k+1}-x_k}^2$ is violated by the approximate solution $\xhat_{k+1}$,  the PG parameter is decreased by setting $\alpha_{k+1} \gets \zeta \alpha_k$ and we set $x_{k+1} \gets x_k$. 

Let us discuss the four termination conditions that we use in Algorithm~\ref{alg:main}.  

\begin{enumerate}
\item \textbf{Maximum time.} Algorithm~\ref{alg:main} is terminated if $12$ hours is reached.
\item \textbf{Maximum iterations.}  Algorithm~\ref{alg:main} is terminated if  $10^6$ iterations is reached.
\item \textbf{Approximate solution found.} Algorithm~\ref{alg:main} is terminated if 
\begin{align}\label{eq:algo1term}
    \frac{\norm{\xhat_{k+1} - x_k} + \sqrt{2\alpha_k\big(\phi(\xhat_{k+1};x_k,\alpha_k)  -  \phid(\yhat_{k+1};x_k,\alpha_k)\big)}}{\min(1, \alpha_k)}
    \leq \epsilon_{\text{tol}}:= 10^{-5},
\end{align}
which we proceed to justify.  Note that strong duality implies that the pair $(\xhat_{k+1}, \yhat_{k+1})$ returned by Algorithm~\ref{alg:dual-gradient-descent} satisfies $\phi(\xhat_{k+1};x_k,\alpha_k) - \phi(\T(x_k, \alpha_k);x_k,\alpha_k) \leq \phi(\xhat_{k+1};x_k,\alpha_k) - \phid(\yhat_{k+1};x_k,\alpha_k)$ so that $\xhat_{k+1}$ is by definition (see \eqref{def:ipg-update}) an $\epsilonbar$-PG update with $\epsilonbar = \phi(\xhat_{k+1};x_k,\alpha_k)  - \phid(\yhat_{k+1};x_k,\alpha_k)$. It follows from this fact, the reverse triangular inequality, and Lemma~\ref{lemma:chibound} that 
$$
\begin{aligned}
    \norm{\T(x_k,\alpha_k)-x_k} - \norm{\xhat_{k+1}-x_k} 
    &\leq \norm{\T(x_k,\alpha_k) - x_k + x_k - \xhat_{k+1}} \\
    &= \norm{\T(x_k,\alpha_k) - \xhat_{k+1}} \leq \sqrt{2\alpha_k\epsilonbar}.
\end{aligned}
$$
Combining this with the definitions of $\chipgk$ and $\epsilonbar$, and \eqref{eq:algo1term}, it follows that
$$
\chipgk  = \frac{\norm{\T(x_k,\alpha_k)-x_k}}{\alpha_k}
\leq \frac{\norm{\xhat_{k+1}-x_k}}{\alpha_k} + \frac{\sqrt{2\alpha_k\epsilonbar}}{\alpha_k}
\leq \frac{\norm{\xhat_{k+1} - x_k} + \sqrt{2\alpha_k\epsilonbar}}{\min(1,\alpha_k)}
\leq \epsilon_{\text{tol}},
$$
i.e., $\xhat_{k}$ is an approximate solution since $\chipgk \leq \epsilon_{\text{tol}}$.  This justifies~\eqref{eq:algo1term}. 
\item \textbf{Numerical difficulties.} Algorithm~\ref{alg:main} is terminated if numerical difficulties are encountered when the subproblem solver Algorithm~\ref{alg:dual-gradient-descent} is called. In particular, while running our initial tests, we occasionally observed $\epsilon_{k}$ to be on the order of $10^{-12}$ when the sequence $\{x_{k}\}$ approached a stationary point.  In this case, Algorithm~\ref{alg:dual-gradient-descent} was sometimes unable to return an approximate primal-dual pair $(\xhat_{k+1}, \yhat_{k+1})$ that satisfied its approximate optimality conditions before reaching its iteration limit of 5000 iterations; the resulting $\xhat_{k+1}$ in this scenario was often dense. To address this numerical challenge, when Algorithm~\ref{alg:dual-gradient-descent} is terminated due to reaching its iteration limit, a reference point is used to (potentially) set additional groups of $\xhat_{k+1}$ to zero. Concretely,
we first define
$$
\text{ref}(k) := \max\{ \kbar : \kbar \leq k \ \text{and Algorithm~\ref{alg:dual-gradient-descent} terminated successfully in iteration $k$} \},
$$
which allows us to define, for each $i\in[\ngrp]$, the ``corrected" iterate
$$
[\xhat_{k+1}^{\text{corrected}}]_{g_i} =
\begin{cases}
[\hat x_{k+1}]_{g_i} & \text{ if } [x_{\text{ref}(k)}]_{g_i} \neq 0, \\
0 & \text{ if } [x_{\text{ref}(k)}]_{g_i} = 0.
\end{cases}
$$
We use the corrected iterate instead of $\xhat_{k+1}$ if it has a better objective value:
$$
\xhat_{k+1} \gets 
\begin{cases}
\xhat_{k+1}^{\text{corrected}} & \text{if $\phi\big(\xhat_{k+1}^{\text{corrected}}; x_k, \alpha_k\big)\leq \phi\big(\xhat_{k+1}; x_k, \alpha_k\big)$,} \\
 \xhat_{k+1} & \text{otherwise}.
\end{cases}
$$
If two correction steps are needed, we terminate for numerical difficulties.
\end{enumerate}

We now discuss Algorithm~\ref{alg:dual-gradient-descent}. The values for the input parameters are $\xi_2 = 0.5$ and $\eta_2 = 10^{-3}$.  Algorithm~\ref{alg:dual-gradient-descent} is terminated when the inequality in Line~\ref{line:pgd-term11} holds, the inequality in Line~\ref{line:pgd-term21} holds, or $5000$ iterations are performed. In the last case, a ``correction" step is attempted as described above. For $k = 0$, we set $\sigma \gets 1$ and $\yhat_{0,0} \gets 0$.  For all future iterations $k \geq 1$, we set $\sigma \gets \sigma_{k-1,t}$ and $\yhat_{k,0} \gets \yhat_{k-1,t}$, where $t$ is the final iteration of Algorithm~\ref{alg:dual-gradient-descent} when called during iteration $k-1$. Finally, when $ \optthree{}$ is considered (see the discussion above in this section), Algorithm~\ref{alg:dual-gradient-descent} is terminated when the inequality in Line~\ref{line:pgd-term11} or Line~\ref{line:pgd-term21} (with the right-hand sides replaced by $\const{}/k^3$) holds since this implies that $\xhat_{k+1}\in\Tepsk(x_k,\alpha_k)$ with $\epsilon_k = \const{}/k^3$, as needed.     

\subsection{Tests and results} \label{sec.results}

We have a Python implementation that is available upon request. All experiments were conducted on the Computational Optimization Research Laboratory (COR@L) cluster at Lehigh University with an AMD Opteron Processor 6128 2.0 GHz CPU.

\subsubsection{Comparison of the different termination conditions}\label{sec:compare-tc} 

We compare the performance of the three algorithm variants $\optone{}$, $\opttwo{}$, and $\optthree{}$. For this test, we generate overlapping groups as described in Section~\ref{sec:problem} using values $\ratio \in \{0.1, 0.2, 0.3\}$, $\grpsize{} \in \{10,100\}$, and $\Lambda \in \{ 0.1\Lambda_{\min}, 0.01\Lambda_{\min}\}$, and the $11$ data sets given in Table~\ref{tab:test-db}.  In total, this gives $132$ test instances comprised of $12$ different group constructions for each of the $11$ data sets. 

We tune values $\gamma_1$ for $\optone{}$, $\gamma_2$ for $\opttwo{}$, and $\const{}$ for $\optthree{}$ on the $72$ test instances that correspond to the smaller data sets a9a, colon-cancer, duke breast-cancer, leukemia, mushrooms, and w8a. We searched for best values for $\gamma_1$ and $\gamma_2$ from the set $\{0.1, 0.2, 0.3, 0.4, 0.5\}$, and a best value for $\const{}$ from the set $\{10^i\}_{i=0}^{4}$ (these sets were chosen based on preliminary testing). When deciding the best values for $\gamma_1$, $\gamma_2$ and $\const{}$, we considered an algorithms performance on the 72 problem instances in terms of CPU time (averaged over three runs to account for randomness in the timings). This procedure resulted in the tuned values $\gamma_1 = 0.2$, $\gamma_2 = 0.5$, and $\const = 1000$.  Empirically, the performances of $\optone{}$ and $\opttwo{}$, which use relative criteria each iteration, were less sensitive to parameter tuning when compared to $\optthree{}$, which uses an absolute criterion each iteration.

Using the above tuned parameters, we tested $\optone{}$, $\opttwo{}$, and $\optthree{}$ across the entire $132$ problem instances.  A summary of the termination statuses (see Section~\ref{sec:implementation}) returned by the three algorithms is found in Table~\ref{tab:full-run}.  These results show that all three variants successful find an approximate solution to the majority of problem instances.     
We note that all three variants terminated due to the maximum iteration limit on all 12 instances of the data set madelon, and that all instances for which an algorithm reached the maximum time limit was for the data set news20.binary.

\begin{table}[!ht]
\centering
\caption{Termination status summary for the algorithm variants $\optone{}$, $\opttwo{}$, and $\optthree{}$ on the $132$ test instances described in Section~\ref{sec:compare-tc}.  See Section~\ref{sec:implementation} for the precise meaning of each termination condition.}
\smallskip
{\begin{tabular}{|c|c|c|c|c|}
\hline
                      & approximate    & maximum         &  maximum    & numerical    \\
                      & solution found & iteration limit &  time limit & difficulties \\ 
\hline
$\optone{}$ & 108                             & 16                     & 7                 & 1                 \\ 
$\opttwo{}$ & 107                             & 15                     & 8                 & 2                 \\ 
$\optthree{}$ & 107                             & 16                     & 9                 & 0                 \\ 
\hline
\end{tabular}}
\label{tab:full-run}
\end{table}

Since Table~\ref{tab:full-run} verifies that all three algorithm variants are relatively robust, we now consider various performance metrics in detail.  First we  compare the computing times.

\medskip
\noindent\textbf{CPU time comparison.}
Figure~\ref{fig:ppflogitnat} contains performance profiles based on~\cite{morales2002numerical} for comparing the computing times of two algorithms on a collection of problem instances. For a plot that compares algorithms ${\tt option\_i}$ and ${\tt option\_j}$ for $\{i,j\}\subset\{1,2,3\}$ with $i \neq j$, each bar corresponds to a problem instance, with the height of the bar given by 
\begin{equation}\label{measure}
-\log_2
\left(
\frac{\text{time required by }{\tt option\_i}}{\text{time required by }{\tt option\_j}}
\right). 
\end{equation}
Thus, an upward pointing (blue) bar indicates that ${\tt option\_i}$ took less time than ${\tt option\_j}$ to solve that test problem.  In contrast, a downward pointing (red) bar means that ${\tt option\_i}$ took more time than ${\tt option\_j}$ to solve that test problem. In either case, the size of the bar indicates the magnitude of the outperformance. If both algorithms fail to solve a problem instance, that instance is not included in the plot.  For problem instances that are successfully solved by only one of the two algorithms, the height of the bar is set according to the following procedure. For each $p\in\{1,2,\cdots,132\}$ and $i\in\{1,2,3\}$, let $t_{p,i}$ denote the time consumed by ${\tt option\_i}$ on the $p$th problem and $\Scal_{i,j}$ denote the set of all problem instances that are successfully solved by both ${\tt option\_i}$ and  ${\tt option\_j}$.  Then, if exactly one of ${\tt option\_i}$ and ${\tt option\_j}$ fails on problem instance $p$, the height of the associated bar is set as $1.5\max_{j\in\Scal_{i,j}}\{|\log_2(t_{p,i}/t_{p,j})|\}$ pointing in the direction of the algorithm that successfully solved the problem. For each given plot, a single performance metric can be computed for each of the two competing algorithms by computing the areas of their respective bars (the width of each bar is one so the the area of each bar is equal to its height); theses areas are indicated in the legend of each plot.  As shown in Fig~\ref{fig:ppflogitnat}, the adaptive criteria used by $\optone{}$ and $\opttwo{}$ outperforms the absolute criterion used in $\optthree{}$. 

\begin{figure}[!ht]
  \begin{center}
    \includegraphics[width=0.32\textwidth]{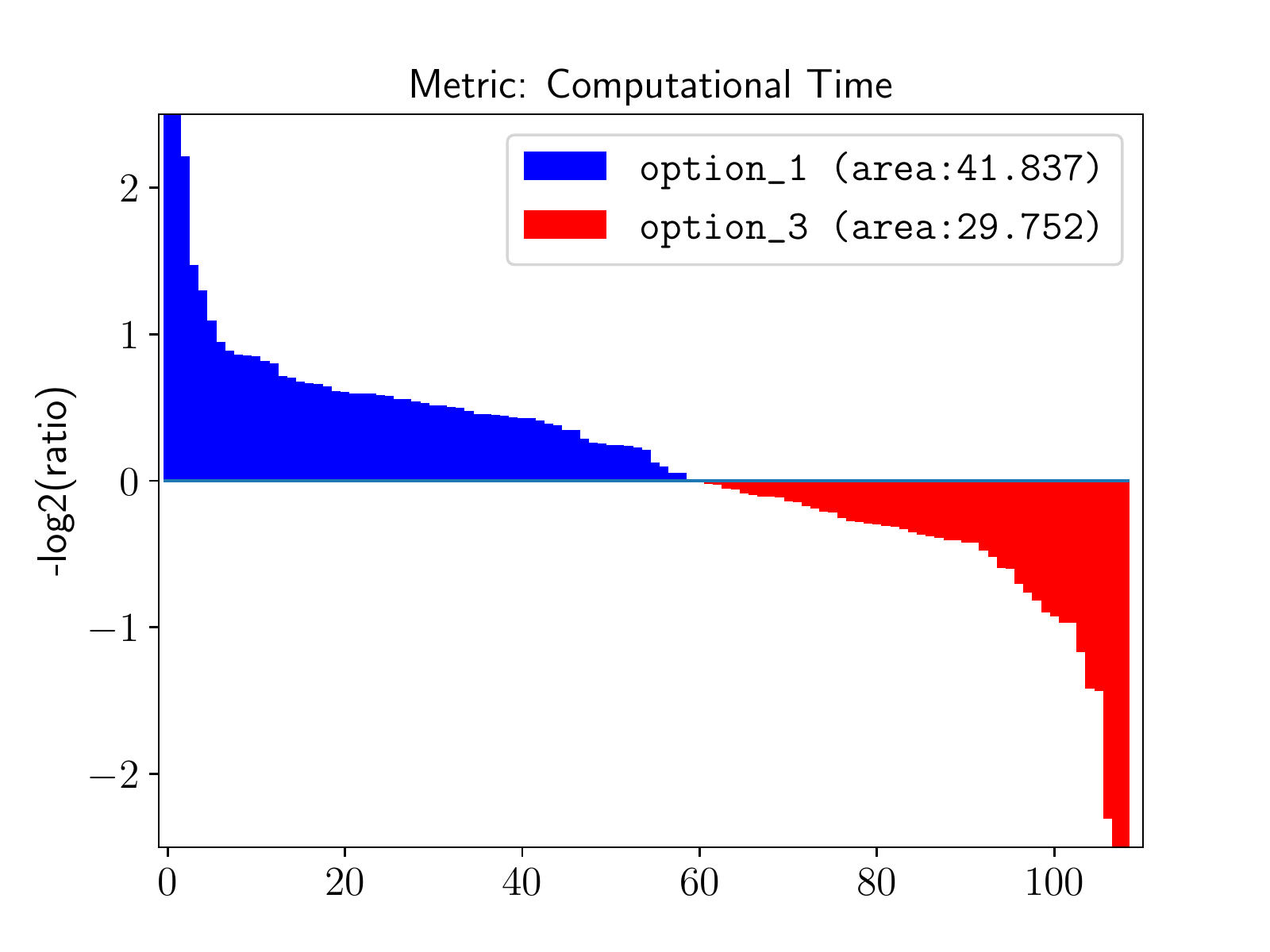}
    \includegraphics[width=0.32\textwidth]{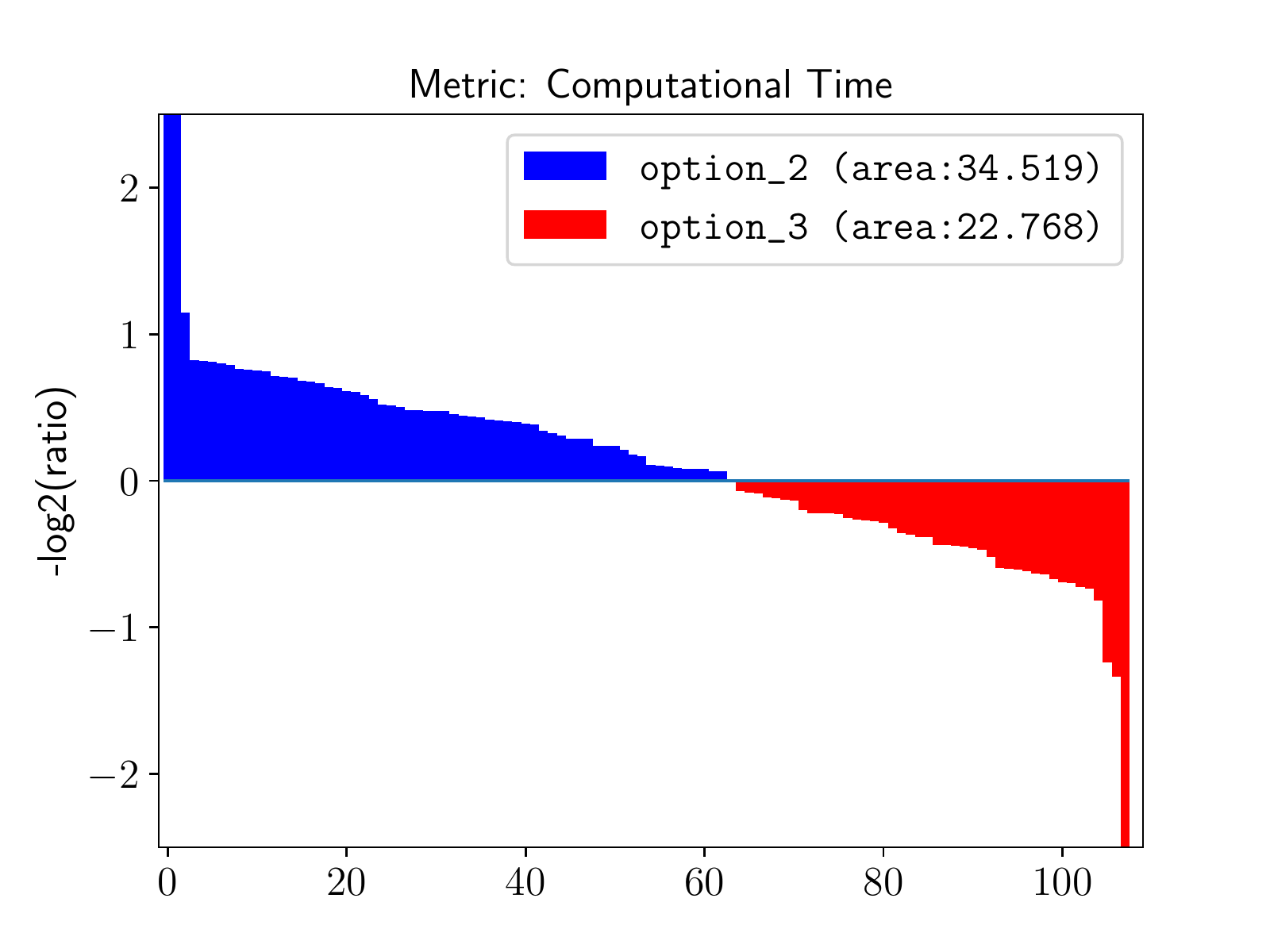}
    \includegraphics[width=0.32\textwidth]{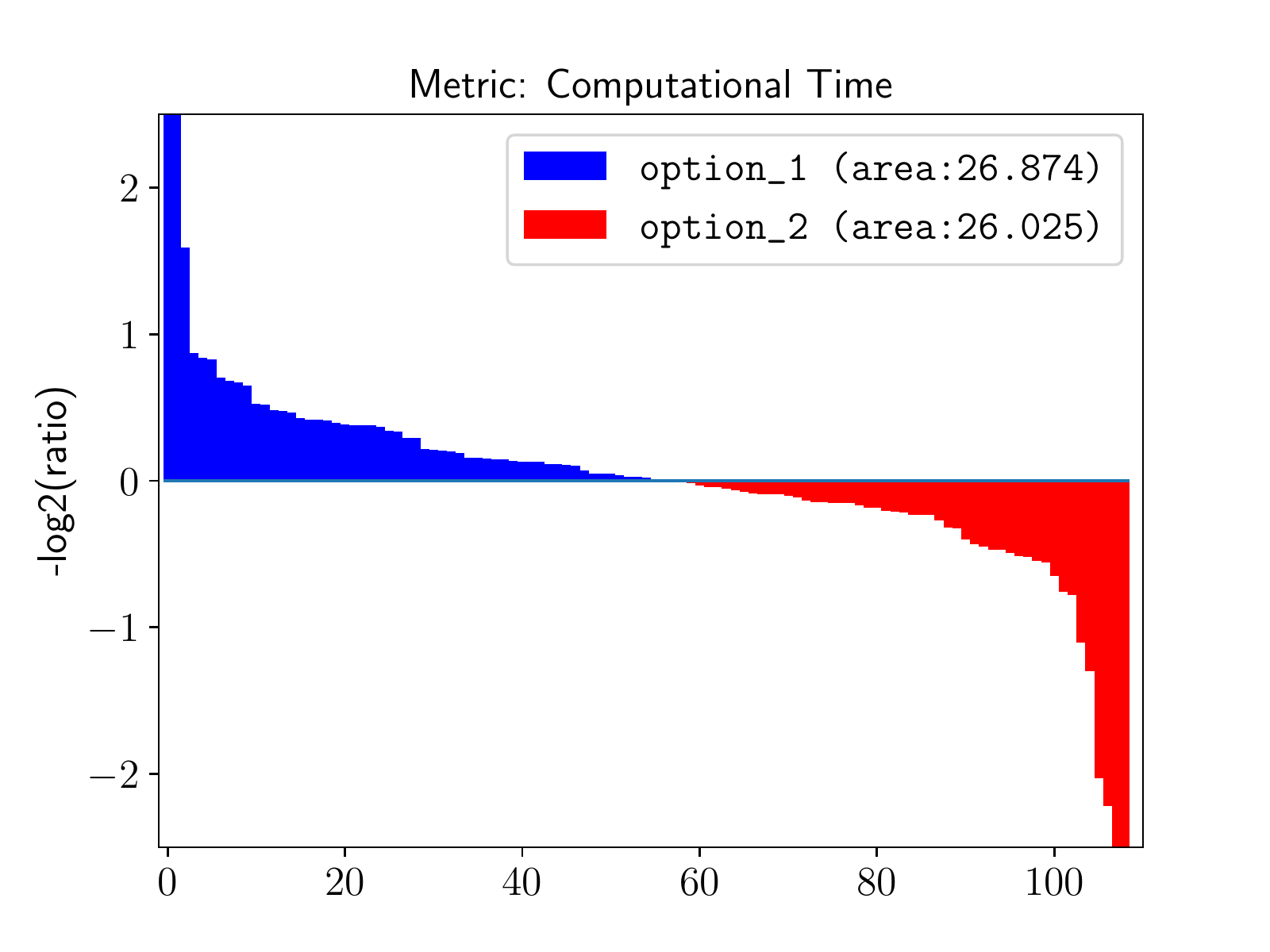}
  \end{center}
  \vspace*{-0.5cm}
  \caption{A performance profile for CPU time (seconds). 
  In each plot, we exclude problem instances for which both algorithms fail.}
  \label{fig:ppflogitnat}
\end{figure}

\noindent
\textbf{Final objective value and sparsity.} 
For problem instance $p\in\{1,2,\dots,132\}$ and algorithm ${\tt option\_i}$ for $i\in\{1,2,3\}$, we let $F_{p,i}$ denote the final objective value returned by ${\tt option\_i}$ when solving problem $p$. If $F_{i} - F_{j} < -10^{-6}$, then we say ${\tt  option\_i}$ obtained a ``better"  objective value compared to ${\tt option\_j}$; if  $F_{i} - F_{j} > 10^{-6}$, then we say ${\tt option\_i}$ obtained a ``worse" objective value compared to ${\tt option\_j}$; and if $|F_{i} - F_{j}|\leq 10^{-6}$, then we consider them to have performed the ``same".

One must also consider the sparsity of the returned solutions. If the solution returned by ${\tt option\_i}$ is sparser than the solution returned by ${\tt option\_j}$, we say that ${\tt option\_i}$ performed ``better" than ${\tt option\_j}$; if the solution returned by ${\tt option\_i}$ is denser than the solution returned by ${\tt option\_j}$, we say that ${\tt option\_i}$ performed ``worse" than ${\tt option\_j}$; if the solutions returned by both algorithms have the same level of sparsity, we say that the two algorithms performed the ``same".

The results from our tests can be found in Table~\ref{tab:sparsity-and-F}. From these results, we may first conclude that the adaptive termination criteria ($\optone{}$ and $\opttwo$) outperform the absolute termination criterion ($\optthree$) in achieving a lower objective value. Second, the algorithms based on an adaptive termination criteria ($\optone{}$ and $\opttwo{}$) slightly outperform the algorithm using an absolute termination criterion ($\optthree{}$) in terms of obtaining sparser solutions.  Third, $\optone{}$ slightly outperforms $\opttwo{}$ in terms of achieving a better final objective value but at the expense of being slightly outperformed by $\opttwo{}$ in terms of final solution sparsity. 
 
\begin{table}[!ht]
\centering
\setlength\tabcolsep{3pt}
\caption{A comparison of solution sparsity and final objective value for $\optone{}$, $\opttwo{}$, and $\optthree{}$ over the entire $132$ problem instances. For example, the $6$ in the first row (``$\optone{}$\; versus \;$\optthree{}$" ) indicates that $\optone{}$ returned a sparser solution (i.e., a ``better" solution) than $\optthree{}$ on 6 problem instances.  Similarly, the $24$ in that same row means that $\optone{}$ returned a lower objective function value (i.e., a ``better" final objective value) than $\optthree{}$ on $24$ problem instances.}
\medskip
{\begin{tabular}{|l|ccc|ccc|}
\hline
                              & \multicolumn{3}{c|}{sparsity}                                                      & \multicolumn{3}{c|}{objective value}                                                          \\ 
\hline
                              & \multicolumn{1}{l}{better} & \multicolumn{1}{l}{same} & \multicolumn{1}{l|}{worse} & \multicolumn{1}{l}{better} & \multicolumn{1}{l}{same} & \multicolumn{1}{l|}{worse} \\ 
\hline 
$\optone{}$\; versus \;$\optthree{}$ & 6                          & 121                      & 5                         & 24                         & 103                      & 5                         \\
$\opttwo{}$\; versus \;$\optthree{}$   & 8                          & 118                      & 6                         & 23                         & 103                      & 6                         \\
$\optone{}$\; versus \;$\opttwo{}$\; & 4                          & 118                      & 10                        & 10                         & 115                      & 7                         \\ 
\hline
\end{tabular}}
\label{tab:sparsity-and-F}
\end{table}

\subsubsection{Subproblem solver comparison: Algorithm~\ref{alg:dual-gradient-descent} versus projected gradient ascent} \label{sec:pgd}

We now explore the support identification property of our subproblem solver Algorithm~\ref{alg:dual-gradient-descent}.  It is clear from the previous section that Algorithm~\ref{alg:main} with any $\opt\in \{\optone{},\opttwo{},\optthree{}\}$ is returns sparse solutions when Algorithm~\ref{alg:dual-gradient-descent} is used as the subproblem solver.  Here, we compare the sparsity of the solutions produced by Algorithm~\ref{alg:main} when our proposed subproblem solver Algorithm~\ref{alg:dual-gradient-descent} is replaced by a projected gradient ascent (PGA) method.  The results of our tests are in Table~\ref{tab:subsolver}; we only compare PGA to $\optone{}$ since $\optone{}$ and $\opttwo{}$ perform similarly. 

For each data set in Table~\ref{tab:test-db} we created a single test instance by first setting $\ratio{} \gets 0.1$ and $\grpsize{} \gets 10$ (see Section~\ref{sec:problem}).  To ensure that a sparse solution existed for each problem instance, we selected $\Lambda$ (see Section~\ref{sec:problem}) differently for each data set as indicated in the  second column of Table~\ref{tab:subsolver}. Note that the three largest data sets from Table~\ref{tab:test-db} are excluded from Table~\ref{tab:subsolver} because the PGA subproblem solver was not able to achieve any reasonable solution within the time limit (see Section~\ref{sec:implementation}).

Table~\ref{tab:subsolver} indicates that when Algorithm~\ref{alg:main} uses either Algorithm~\ref{alg:dual-gradient-descent} or the PGA method as the subproblem solver the same objective value (up to 6 digits of accuracy) is achieved.  However, while Algorithm~\ref{alg:dual-gradient-descent} consistently finds sparse solutions, the subproblem solver PGA consistently finds dense solutions. Although it is possible to enhance the PGA method through a post-processing step that forces blocks of variables with small magnitude to zero, this would require the careful tuning of a threshold parameter.  Our method does not require any such post-processing to achieve sparse solutions.  Although not the focus of this section, we  observed that Algorithm~\ref{alg:main} converged in fewer iterations when Algorithm~\ref{alg:dual-gradient-descent} is used as the subproblem solver compared to when PGA is used.  It is clear that using Algorithm~\ref{alg:dual-gradient-descent} is crucial to obtaining sparse solutions. 

\begin{table}[!ht]
\centering
\setlength\tabcolsep{5pt}
\caption{The test results for  Algorithm~\ref{alg:main} when using  Algorithm~\ref{alg:dual-gradient-descent} or the PGA algorithm as its subproblem solvers.  Columns ``\#z", ``\#nz", and ``$F$" give the number of zero groups, the number of non-zero groups, and the final objective value, respectively.}
\medskip
{\begin{tabular}{|cc|ccc|ccc|}
\hline
                    &              &  \multicolumn{3}{c|}{Algorithm~\ref{alg:dual-gradient-descent} ($\optone{}$)} & \multicolumn{3}{c|}{PGA} \\
           data set &    $\Lambda$ &               \#z  &  \#nz  &  $F$ & \#z & \#nz & $F$ \\
\hline
                a9a &  0.013458 &  $12$ &  $2$ & $0.508337$ & $0$ &  $14$ & $0.508337$ \\
       colon-cancer &  0.017751 & $213$ & $10$ & $0.336270$ & $1$ & $222$ & $0.336270$ \\
 duke breast-cancer &  0.016198 & $779$ & $13$ & $0.246910$ & $2$ & $790$ & $0.246910$ \\
            gisette &  0.012003 & $536$ & $20$ & $0.402671$ & $2$ & $554$ & $0.402671$ \\
           leukemia &  0.020514 & $781$ & $11$ & $0.258627$ & $0$ & $792$ & $0.258627$ \\
            madelon &  0.000402 &  $19$ & $37$ & $0.666079$ & $0$ &  $56$ & $0.666112$ \\ 
          mushrooms &  0.009528 &  $10$ &  $3$ & $0.316138$ & $0$ &  $13$ & $0.316138$ \\
                w8a &  0.006687 &  $24$ & $10$ & $0.429029$ & $0$ &  $34$ & $0.429029$ \\
\hline
\end{tabular}}
\label{tab:subsolver}
\end{table}

\section{Conclusion}\label{sec.conclusion}

We proposed a PG framework to solve optimization problems that use a regularizer whose proximal operator does not have a closed-form solution. To address this challenge, we proposed two adaptive criteria for approximately solving the PG subproblem, and then proved a worst-case complexity result for the maximum number of iterations before an approximate first-order solution is computed. We also designed an enhanced solver for the PG subproblem when the \grplone{} regularizer is used.  When this subproblem solver is used within our overall PG method, we provide an upper bound on the number of iterations before optimal support identification occurs. Numerical experiments on regularized logistic regression problems illustrate the effectiveness of our approach for efficiently finding structured sparse solutions.

%

\bibliographystyle{plain}
\bibliography{references}

\end{document}